%% file: main.tex
\documentclass[final,3p,times]{elsarticle}

\input{packagesAndMacros}

\begin{document}

\begin{frontmatter}

\title{ Efficient solvers for hybridized three-field mixed finite element coupled poromechanics\tnoteref{t1}}
\tnotetext[t1]{This work is a collaborative effort.}

\author[UniPD]{Matteo Frigo}
\ead{matteo.frigo.3@phd.unipd.it}

\author[LLNL]{Nicola Castelletto}
\ead{castelletto1@llnl.gov}

\author[UniPD]{Massimiliano Ferronato}
\ead{massimiliano.ferronato@unipd.it}

\author[LLNL]{Joshua A. White}
\ead{jawhite@llnl.gov}

\address[UniPD]{Department of Civil, Environmental and Architectural Engineering, University of Padova, Italy}
\address[LLNL]{Atmospheric, Earth, and Energy Division, Lawrence Livermore National Laboratory, United States}

\journal{arXiv}

\begin{abstract}
We consider a mixed hybrid finite element formulation for coupled poromechanics. A stabilization strategy based on a macro-element approach is advanced to eliminate the spurious pressure modes appearing in undrained/incompressible conditions. The efficient solution of the stabilized mixed hybrid block system is addressed by developing a class of block triangular preconditioners based on a Schur-complement approximation strategy. Robustness, computational efficiency and scalability of the proposed approach are theoretically discussed and tested using challenging benchmark problems on massively parallel architectures.
\end{abstract}


\end{frontmatter}



\allowdisplaybreaks{}

\section{Introduction}

We focus on a three-field mixed (displacement-velocity-pressure) formulation of classical linear poroelasticity \cite{Bio41,Cou04}.
Let $\Omega \subset \mathbb{R}^d$ ($d=2,3$) and $\Gamma$ be the domain occupied by the porous medium and its Lipschitz-continuous boundary, respectively, with $\tensorOne{x}$ the position vector in $\mathbb{R}^d$. 
We denote time with $t$, belonging to an open interval $\mathcal{I} =\left( 0, t_{\max} \right)$.
The boundary is decomposed as $\Gamma = \overline{\Gamma_{u} \cup \Gamma_{\sigma}} = \overline{\Gamma_p \cup \Gamma_{q}}$, with $\Gamma_u \cap \Gamma_{\sigma} = \Gamma_p \cap \Gamma_{q} = \emptyset$, and $\tensorOne{n}$ denotes its outer normal vector.
Assuming quasi-static, saturated, single-phase flow of a slightly compressible fluid, the set of governing equations consists of a conservation law of linear momentum and a conservation law of mass expressed in mixed form, i.e., introducing Darcy's velocity as an additional unknown.
The strong form of the initial-boundary value problem (IBVP) consists of finding the displacement $\tensorOne{u} : \overline{\Omega} \times \mathcal{I} \rightarrow \mathbb{R}^d$, the Darcy velocity $\tensorOne{q} : \overline{\Omega} \times \mathcal{I} \rightarrow \mathbb{R}^d$, and the excess pore pressure $p : \overline{\Omega} \times \mathcal{I} \rightarrow \mathbb{R}$ that satisfy:
\begin{linenomath}
\begin{subequations}
\begin{align}
  \nabla \cdot \left( \tensorFour{C}_{dr} : \nabla^{s} \tensorOne{u} - b p \tensorTwo{1} \right)
  &=
  \tensorOne{0} & &\mbox{ in } \Omega \times \mathcal{I}
  & &\mbox{(equilibrium)}, \label{momentumBalanceS}\\
  \mu \tensorTwo{\kappa}^{-1} \cdot \tensorOne{q} + \nabla p
  &= 
  \tensorOne{0} & &\mbox{ in } \Omega \times \mathcal{I}
  & &\mbox{(Darcy's law)}, \label{darcyS}  \\	
  b \nabla \cdot \dot{\tensorOne{u}} + S_{\epsilon} \dot{p} + \nabla \cdot \tensorOne{q} &=
  f & &\mbox{ in } \Omega \times \mathcal{I}
  & &\mbox{(continuity)}. \label{massBalanceS}
\end{align}\label{eq:IBVP}\null
\end{subequations}
\end{linenomath}
Here, $\tensorFour{C}_{dr}$ is the rank-four elasticity tensor, $\nabla^{s}$ is the symmetric gradient operator, $b$ is the Biot coefficient, and $\tensorTwo{1}$ is the rank-two identity tensor; $\mu$ and $\tensorTwo{\kappa}$ are the fluid viscosity and the rank-two permeability tensor, respectively; $S_{\epsilon}$ is the constrained specific storage coefficient, i.e. the reciprocal of Biot's modulus, and $f$ the fluid source term. 
The following set of boundary and initial conditions complete the formulation:
\begin{linenomath}
\begin{subequations}
\begin{align}
  \tensorOne{u} &= \bar{\tensorOne{u}}
  & &\mbox{ on } \Gamma_u \times \mathcal{I}, & \label{momentumBalanceS_DIR}\\
  \left( \tensorFour{C}_{dr} : \nabla^{s} \tensorOne{u} - b p \tensorTwo{1} \right) \cdot \tensorOne{n} &= 
  \bar{\tensorOne{t}}
  & &\mbox{ on } \Gamma_{\sigma} \times \mathcal{I}, & \label{momentumBalanceS_NEU}\\
  \tensorOne{q} \cdot \tensorOne{n} &= \bar{q}
  & &\mbox{ on } \Gamma_{q} \times \mathcal{I}, &\label{massBalanceS_NEU}\\    
  p &=\bar{p}
  & &\mbox{ on } \Gamma_p \times \mathcal{I}, & \label{massBalanceS_DIR}\\
  p(\tensorOne{x}, 0) &=p_0
  & &\mbox{ } \tensorOne{x} \in \overline{\Omega}, &  \label{massBalanceS_IC}
\end{align}\label{eq:IBVP_b}\null
\end{subequations}
\end{linenomath}
where $\bar{\tensorOne{u}}$, $\bar{\tensorOne{t}}$, $\bar{q}$, and $\bar{p}$ are the prescribed boundary displacements, tractions, Darcy velocity and excess pore pressure, respectively, whereas $p_0$ is the initial excess pore pressure.
More precisely, the initial condition should be given as
\begin{equation}
  b \nabla \cdot \tensorOne{u} (\tensorOne{x}, 0) + S_{\epsilon} p (\tensorOne{x}, 0) = b \nabla \cdot {\tensorOne{u}_0 } + S_{\epsilon} {p_0},
  \qquad
  \tensorOne{x} \in \overline{\Omega},
\end{equation}
with $\tensorOne{u}_0$ the initial displacement field, i.e. specifying the initial fluid content of the medium \cite{Bio41}.
However, in practical simulations, the pressure is often measured or computed through the hydrostatic assumption, and the initial displacement is then obtained so as to satisfy Equation \eqref{momentumBalanceS}---see, e.g. \cite{GirKumWhe16,GirWheAlmDan19}.
We refer the reader to \cite{Sho00} for a rigorous discussion on this issue.

Let us denote with $\boldsymbol{H}^1(\Omega)$ the Sobolev space of vector functions whose first derivatives are square-integrable, i.e., they belong to the Lebesgue space $L^2(\Omega)$; and let $\boldsymbol{H}(\text{div};\Omega)$ be the Sobolev space of vector functions with square-integrable divergence.
Introducing the spaces:
\begin{linenomath}
\begin{subequations}
\begin{align}
  \boldsymbol{\mathcal{U}} &= \{ \tensorOne{u} \in \boldsymbol{H}^1(\Omega) \ \ | \ \
  \tensorOne{u}|_{\Gamma_u}=\bar{\tensorOne{u}} \}, &
  \boldsymbol{\mathcal{U}}_0 &= \{ \tensorOne{u} \in \boldsymbol{H}^1(\Omega) \ \ | \ \
  \tensorOne{u}|_{\Gamma_u}=\tensorOne{0} \}, \label{eq:spaceH1_0}  \\
  \boldsymbol{\mathcal{Q}} &= \{\tensorOne{q} \in \boldsymbol{H}(\text{div};\Omega) \ \ | \ \
  \tensorOne{q} \cdot \tensorOne{n}|_{\Gamma_q}=\bar{q} \}, &
  \boldsymbol{\mathcal{Q}}_0 &= \{\tensorOne{q} \in \boldsymbol{H}(\text{div};\Omega) \ \ | \ \
  \tensorOne{q} \cdot \tensorOne{n}|_{\Gamma_q}=0 \}, \label{eq:spaceHdiv0}   \\
  \mathcal{P} &= \{p \in L^2(\Omega) \}, \label{eq:spaceL2}   
\end{align}\label{eq:functionSpaces}
\end{subequations}
\end{linenomath}
the weak form of the IBVP \eqref{eq:IBVP} reads: find $\{ \tensorOne{u}(t),\tensorOne{q}(t), p(t) \} \in \boldsymbol{\mathcal{U}} \times \boldsymbol{\mathcal{Q}} \times \mathcal{P}$ such that $\forall t \in \mathcal{I}$: 
\begin{linenomath}
\begin{subequations}
\begin{align}
  &{(\nabla^s \tensorTwo{\eta}, \tensorFour{C}_{\text{dr}}:\nabla^s \tensorOne{u})}_{\Omega}-{(\text{div} \; \tensorTwo{\eta},bp)}_{\Omega} 
  ={ ( \tensorTwo{\eta}, \bar{\tensorOne{t}} ) }_{\Gamma_\sigma} &&\forall \tensorTwo{\eta} \in \boldsymbol{\mathcal{U}}_0 , \label{momentumBalanceW}\\
  &{(\tensorTwo{\phi}, \mu \tensorTwo{\kappa}^{-1} \cdot \tensorOne{q})}_{\Omega} - {(\text{div} \; \tensorTwo{\phi}, p)}_{\Omega} 
  = - {(\tensorTwo{\phi} \cdot \tensorOne{n}, \bar{p})}_{\Gamma_p}  &&\forall \tensorTwo{\phi} \in \boldsymbol{\mathcal{Q}}_0, \label{darcyW} \\
  &{(\chi, b \; \text{div} \; \dot{\tensorOne{u}})}_{\Omega} + {(\chi, \text{div} \; \tensorOne{q})}_{\Omega} + {(\chi, S_{\epsilon} \dot{p})}_{\Omega} = 
  {(\chi,f)}_{\Omega} &&\forall \chi \in L^2(\Omega),\label{massBalanceW}
\end{align}\label{eq:IBVP_cW}\null
\end{subequations}
\end{linenomath}
where ${(\cdot,\cdot)}_{\Omega}$ denote the inner products of scalar functions in $L^2(\Omega)$, vector functions in ${[L^2(\Omega)]}^d$, or second-order tensor functions in ${[L^2(\Omega)]}^{d \times d}$, as appropriate, and ${(\cdot, \cdot)}_{\Gamma_{*}}$ denote the inner products of scalar functions or vector functions on the boundary $\Gamma_{*}$.
For the analysis of the well-posedness of the Biot continuous problem \eqref{eq:IBVP} in weak form \eqref{eq:IBVP_cW} based on the displacement-velocity-pressure formulation, the reader is referred to~\cite{Lip02}.

A widely-used discrete version of the weak form~\eqref{eq:IBVP_cW} is based on low-order elements.
Precisely, lowest-order continuous ($\mathbb{Q}_1$), lowest-order Raviart-Thomas ($\mathbb{RT}_0$), and piecewise constant ($\mathbb{P}_0$) spaces are often used for the approximation of displacement, Darcy’s velocity, and fluid pore pressure, respectively.
The attractive features of this choice are element-wise mass conservation and robustness with respect to highly heterogeneous hydromechanical properties, such as high-contrast permeability fields typically encountered in real-world applications.
Another attractive feature stems from the hybridization of the mixed three-field formulation, as proposed for instance in~\cite{NiuRuiSun19}.
The hybridized formulation is obtained by (i) considering one degree of freedom per edge/face per element for the normal component of Darcy's velocity, and (ii) introducing one Lagrange  multiplier on edges/faces in the computational mesh, i.e. an interface pressure, to enforce velocity continuity.
The main advantage of the hybrid formulation is that it is amenable to static condensation.

Unfortunately, the $\mathbb{Q}_1$-$\mathbb{RT}_0$-$\mathbb{P}_0$ discretization spaces do not intrinsically satisfy the LBB condition in the undrained/incompressible limit~\cite{Lip02,HagOsnLan12b}.
This can result in spurious modes for the pressure, with non physical oscillations of the discrete solution.
Different stabilization strategies have been proposed in the literature.
They can be
essentially classified in two groups based on whether they: (1) enrich the discretization spaces to guarantee the LBB condition, or (2) introduce a proper stabilization term to restore the saddle-point problem solvability. In the context of Biot's poroelasticity, the first strategy is followed for instance in~\cite{Rod_etal18,NiuRuiHu19}, where proper bubble functions are used to enlarge the space used for the displacement approximation.
This is a mathematically elegant and robust stabilization technique, but it can negatively impact the algebraic structure of the resulting discrete problem, with a possible degradation of the solver computational efficiency.
The second approach, used for instance in~\cite{HonMalFerJan18,CamWhiBor19}, has a much smaller impact on the algebraic structure of the problem, but depends on the choice of appropriate stabilization coefficients that typically introduce some numerical diffusion. Such coefficients should be properly tuned to guarantee the stabilization effectiveness with no detrimental effect on the solution accuracy.  
In this work, we adopt the second approach by proposing a local pressure jump stabilization technique based on a macro-element approach \cite{SilKec90,CamWhiBor19} that is applicable to both mixed and mixed hybrid formulations.

Then, we concentrate on the efficient solution of the non-symmetric algebraic systems obtained by the application of the stabilized formulation.
Several strategies have been already developed for the two-field and mixed three-field formulations, with most of the efforts towards preconditioned Krylov solvers and multigrid methods
\cite{Lip02,Kuz_etal03,Gas_etal04,BerFerGam07,BerFerGam08,FerCasGam10,WhiBor11,AxeBlaByc12,TurArb14,
Luo_etal15,CasWhiFer16,GasRod17,Luo_etal17,LeeMarWin17,Adl_etal18,HonKra18,
FerFraJanCasTch19,FriCasFer19,Adl_etal20,Bui_etal20}. 
Some authors also focused on
sequential-implicit approaches \cite{JhaJua07,KimTchJua11a,KimTchJua11b,MikWhe13,GirKumWhe16,Alm_etal16,
Bot_etal17,Bor_etal18,DanGanWhe18,DanWhe18,Hon_etal18}, where the discrete poromechanical equilibrium equation and the Darcy flow sub-problem are addressed independently, iterating until convergence.
Here, a class of block-triangular preconditioners for accelerating the iterative convergence by Krylov subspace methods is proposed for the stabilized mixed hybrid approach.
We prove that the hybridization of the classical three-field mixed formulation brings better algebraic properties for the resulting discrete problem, which are exploited by the proposed iterative solver.
Performance and robustness of the algorithms are demonstrated in weak and strong scaling studies including both theoretical and field application benchmarks.
Finally, a few concluding remarks close the presentation.

\section{Fully-discrete model}
Let us consider a non-overlapping partition $\mathcal{T}_h$ of the domain $\Omega$ consisting of $n_T$ quadrilateral ($d=2$) or hexahedral ($d=3$) elements.
Let $\mathcal{E}_h$ be the collection of edges ($d=2$) or faces ($d=3$) of elements $T \in \mathcal{T}_h$. Denote with $\tensorOne{n}_e$ the outer normal vector from $e \in \partial T$, where $\partial T$ is the collection of the edges or faces belonging to $T$.
Time integration is performed with the Backward Euler method.
The interval $\mathcal{I}$ is partitioned into $N$ subintervals $\mathcal{I}_n=(t_{n-1},t_n)$, $n = 1,\ldots,N$, where $\Delta t=t_n-t_{n-1}$.

\subsection{Mixed Finite Element (MFE) Method}
\label{sec:MFE}
First, we define the finite dimensional counterpart of the spaces given in \eqref{eq:functionSpaces}:
\begin{linenomath}
\begin{subequations}
\begin{align}
  \boldsymbol{\mathcal{U}}^h &= \{ \tensorOne{u}^h \in \boldsymbol{\mathcal{U}} \ \ | \ \
  \tensorOne{u}^h|_T \in {[\mathbb{Q}_1(T)]}^{d}, \; \forall T \in \mathcal{T}_h \}, &
  \boldsymbol{\mathcal{U}}^h_0 &= \{ \tensorOne{u}^h \in \boldsymbol{\mathcal{U}}_0 \ \ | \ \
  \tensorOne{u}^h|_T \in {[\mathbb{Q}_1(T)]}^{d}, \; \forall T \in \mathcal{T}_h \}, \label{eq:spaceU_h0} \\
  \boldsymbol{\mathcal{Q}}^h &= \{\tensorOne{q}^h \in \boldsymbol{\mathcal{Q}} \ \ | \ \
  \tensorOne{q}^h|_T \in {[\mathbb{RT}_0(T)]}, \; \forall T \in \mathcal{T}_h \}, &
  \boldsymbol{\mathcal{Q}}^h_0 &= \{\tensorOne{q}^h \in \boldsymbol{\mathcal{Q}}_0 \ \ | \ \
  \tensorOne{q}^h|_T \in {[\mathbb{RT}_0(T)]}, \; \forall T \in \mathcal{T}_h \},  \label{eq:spaceQ_h0}\\
  \mathcal{P}^h &=\{ p^h \in L^{2} \ \ | \ \
  p^h|_T \in {[\mathbb{P}_0(T)]}, \; \forall T \in \mathcal{T}_h \}, \label{eq:spaceP_h}
\end{align}\label{eq:functionSpaces_D}\null
\end{subequations}
\end{linenomath}
with $\mathbb{Q}_1(T)$ the mapping to $T$ of the space of bilinear polynomials on the unit square ($d=2$) or trilinear polynomials on the unit cube ($d=3$), $\mathbb{RT}_0(T)$ the lowest-order Raviart-Thomas space and $\mathbb{P}_0(T)$ the space of piecewise constant functions in $T$.
Using the definitions~\eqref{eq:functionSpaces_D}, the fully discrete weak form of the IBVP \eqref{eq:IBVP} may be stated as follows: given $\{\tensorOne{u}_0, \tensorOne{q}_0, p_0\} $, find $\{\tensorOne{u}^h_n, \tensorOne{q}^h_n, p^h_n\} \in \boldsymbol{\mathcal{U}}^h \times \boldsymbol{\mathcal{Q}}^h \times \mathcal{P}^h$ such that for $n=\{1,\ldots,N\}$
\begin{linenomath}
\begin{subequations}
\begin{align}
  &{(\nabla^s \tensorTwo{\eta}^{h}, \tensorFour{C}_{\text{dr}}:\nabla^s \tensorOne{u}^h_n)}_{\Omega}
   -{(\text{div}\;\tensorTwo{\eta}^{h},bp^h_n)}_{\Omega} 
   ={(\tensorTwo{\eta}^h, \bar{\tensorOne{t}}_n)}_{\Gamma_\sigma} &&\forall \tensorTwo{\eta}^{h} \in \boldsymbol{\mathcal{U}}^h_0,\\
  &{(\tensorTwo{\phi}^{h}, \mu \tensorTwo{\kappa}^{-1} \cdot \tensorOne{q}^h_n)}_{\Omega} - {(\text{div} \; \tensorTwo{\phi}^{h}, p^h_n)}_{\Omega} 
   = - {(\tensorTwo{\phi}^{h} \cdot \tensorOne{n}, \bar{p}_n)}_{\Gamma_p}  &&\forall \tensorTwo{\phi}^{h} \in \boldsymbol{\mathcal{Q}}^h_0, \\
  &{(\chi^{h}, b \; \text{div} \; {\tensorOne{u}}^h_n)}_{\Omega} +\Delta t {(\chi^{h}, \text{div} \; \tensorOne{q}^h_n)}_{\Omega} 
   + {(\chi^{h}, S_{\epsilon} {p}^h_n)}_{\Omega} = {(\chi^{h},\tilde{f}_n)}_{\Omega} &&\forall \chi^{h} \in \mathcal{P}^h,
\end{align}\label{eq:IBVP_W}\null{}
\end{subequations}
\end{linenomath}
where $\tilde{f}_n =b \; \text{div} \; {\tensorOne{u}}^h_{n-1} + S_{\epsilon} {p}^h_{n-1} + \Delta t f_n$.

Let $\{ \tensorOne{\eta}_i \}_{i \in \mathcal{N}_u \cup \overline{\mathcal{N}}_u}$ be the standard vector nodal basis functions for $\boldsymbol{\mathcal{U}}^h$, with $\mathcal{N}_u$ and $\overline{\mathcal{N}}_u$ the set of indices of basis function vanishing on $\Gamma_u$ and having support on $\Gamma_u$, respectively.
Let $\{ \tensorOne{\phi}_j \}_{j \in \mathcal{N}_q \cup \overline{\mathcal{N}}_q}$ be the edge-/face-based basis functions for $\boldsymbol{\mathcal{Q}}^h$, with $\mathcal{N}_q$ and $\overline{\mathcal{N}}_q$ the set of indices of basis functions vanishing on $\Gamma_q$ and having support on $\Gamma_q$, respectively.
Let $\{ \chi_k \}_{k \in \mathcal{N}_p}$ be the basis for $\mathcal{P}^h$, with $\chi_k$ the characteristic function of the $k$-th element $T_k \in \mathcal{T}_h$ such that $\chi_k(\tensorOne{x}) = 1$ if $\tensorOne{x} \in T_k$, $\chi_k(\tensorOne{x}) = 0$ if $\tensorOne{x} \notin T_k$.
Thus, discrete approximations for the displacement, Darcy's velocity, and pressure are expressed as
\begin{linenomath}
\begin{align}
  \tensorOne{u}^h_n(\tensorOne{x}) &=
  \underbrace{\sum_{i \in \mathcal{N}_u } \tensorOne{\eta}_i(\tensorOne{x}) u_{i,n}}_{:= \mathring{\tensorOne{u}}^h_n} +
  \underbrace{\sum_{i \in \overline{\mathcal{N}}_u } \tensorOne{\eta}_i(\tensorOne{x}) \bar{u}_{i,n}}_{:= \bar{\tensorOne{u}}^h_n},
  &546
  \tensorOne{q}^h_n(\tensorOne{x}) &=
  \underbrace{\sum_{j \in \mathcal{N}_q } \tensorOne{\phi}_j(\tensorOne{x}) q_{j,n}}_{:= \mathring{\tensorOne{q}}^h_n} +
  \underbrace{\sum_{j \in \overline{\mathcal{N}}_q } \tensorOne{\phi}_j(\tensorOne{x}) \bar{q}_{j,n}}_{:= \bar{\tensorOne{q}}^h_n},
  &
  p^h_n(\tensorOne{x}) &=
  \sum_{k \in \mathcal{N}_{p} } {\chi_k}(\tensorOne{x}) p_{k,n}.
  \label{eq:u_q_p_h}
\end{align}
\end{linenomath}
The unknown nodal displacement components $\{u_{i,n}\}$, edge-/face-centered Darcy's velocity components $\{ q_{j,n} \}$, and cell-centered pressures $\{ p_{k,n} \}$ at time level $t_n$ are collected in vectors $\Vec{u}_n \in \mathbb{R}^{n_u}$, $\Vec{q}_n \in \mathbb{R}^{n_q}$, and $\Vec{p}_n \in \mathbb{R}^{n_p}$, with $n_u = |\mathcal{N}_u|$, $n_q = |\mathcal{N}_q|$, and $n_p = |\mathcal{N}_p| = n_T$.
Note that $\tensorOne{u}^h_n$ is given as superposition of function $\mathring{\tensorOne{u}}^h_n$, which honors homogeneous Dirichlet conditions on $\Gamma_\vec{u} \times \mathcal{I}$, and function $\bar{\tensorOne{u}}^h_n$, which provides a lifting of an approximation of the displacement Dirichlet boundary datum \eqref{momentumBalanceS_DIR}.
A similar superposition is used to express $\tensorOne{q}^h_n$.
Hence, $\{ \tensorOne{\eta}_i \}_{i \in \mathcal{N}_u }$ and $\{ \tensorOne{\phi}_j \}_{j \in \mathcal{N}_q}$ are a basis for $\boldsymbol{\mathcal{U}}^h_0$ and $\boldsymbol{\mathcal{Q}}^h_0$, respectively.

Requiring that $\{ \tensorOne{u}^h_n, \tensorOne{q}^h_n, p^h_n \}$ given in \eqref{eq:u_q_p_h} satisfy \eqref{eq:IBVP_W} for each basis function of $\boldsymbol{\mathcal{U}}^h_0$, $\boldsymbol{\mathcal{Q}}^h_0$, and $\mathcal{P}^h$ yields the matrix form of variational problem \cite{CasWhiFer16}:
\begin{linenomath}
\begin{align}\label{eq:mfeSys}
  \blkMat{A}_{M} \blkVec{x} &= \blkMat{b}
  \quad
  \text{with}
  \quad
  \blkMat{A}_{M} = 
  \begin{bmatrix}
    \Mat{A}\sub{uu} &      0          & \Mat{A}\sub{up} \\  
         0          & \Mat{A}\sub{qq} & \Mat{A}\sub{qp} \\
    \Mat{A}\sub{pu} & \Delta t \Mat{A}\sub{pq} & \Mat{A}\sub{pp}
  \end{bmatrix}, 
  \quad
  \blkVec{x} =
  \begin{bmatrix}
    \Vec{u}_n \\
    \Vec{q}_n \\
    \Vec{p}_n
  \end{bmatrix},
  \quad
  \blkVec{b}=
  \begin{bmatrix}
    \Vec{f}_u \\
    \Vec{f}_{q} \\
    \Vec{f}_p
  \end{bmatrix},
\end{align}
\end{linenomath}
where $\Mat{A}\sub{pu} = - \Mat{A}\sub{up}^T$ and $\Mat{A}\sub{pq} = - \Mat{A}\sub{qp}^T$.
Note that $A_{uu} \in \mathbb{R}^{n_u \times n_u}$ and $A_{qq} \in \mathbb{R}^{n_q \times n_q}$ are symmetric and positive definite (SPD) matrices, whereas $A_{pp} \in \mathbb{R}^{n_p \times n_p}$ is a diagonal matrix  with non negative entries.
The explicit expressions of matrices in \eqref{eq:mfeSys} are given in~\ref{app:mat}.

\subsection{Mixed Hybrid Finite Element (MHFE) Method}
\label{sec:MHFE}
The mixed hybrid finite element formulation is obtained by using discontinuous piecewise polynomial functions for Darcy's velocity and enforcing the continuity of the normal fluxes along inter-element edges or faces with the aid of Lagrange multipliers.
We introduce the finite-dimensional Sobolev spaces:
\begin{linenomath}
\begin{subequations}
\begin{align}
  \boldsymbol{\mathcal{W}}^h &= \{\tensorOne{w}^h \in [L^{2}(\Omega)]^d \ \ | \ \
   \tensorOne{w}^h|_T \in {[\mathbb{RT}_0(T)]}, \; \forall T \in \mathcal{T}_h  \}\label{eq:spaceW}  \\
  \boldsymbol{\mathcal{B}}^h &=\{ \pi^h \in L^{2}(e) \ \ | \ \
   \pi^h|_{\Gamma_p}=\bar{p}, \; \pi^h|_e \in {[\mathbb{P}_0(e)]}, \; \forall e \in \mathcal{E}_h \}\label{eq:spaceB} \\
  \boldsymbol{\mathcal{B}}^h_0 &=\{ \pi^h \in L^{2}(e) \ \ | \ \
   \pi^h|_{\Gamma_p}=0, \; \pi^h|_e \in {[\mathbb{P}_0(e)]}, \; \forall e \in \mathcal{E}_h \}\label{eq:spaceB0}
\end{align}\label{eq:functionSpaces_DH}\null
\end{subequations}
\end{linenomath}
where $L^2(e)$ denotes the set of square integrable functions on the element edge or face $e$.
Hence, the fully-discrete variational problem now becomes: given $\{\tensorOne{u}_0, \tensorOne{w}_0, p_0, \pi_0\} $, find $\{\tensorOne{u}^h_n, \tensorOne{w}^h_n, p^h_n, \pi^h_n \} \in \boldsymbol{\mathcal{U}}^h \times \boldsymbol{\mathcal{W}}^h \times \mathcal{P}^h \times \mathcal{B}^h$ such that for $n=\{1,\ldots,N\}$
\begin{linenomath}
\begin{subequations}
\begin{align}
  &{(\nabla^s \tensorTwo{\eta}^h, \tensorFour{C}_{\text{dr}}:\nabla^s \tensorOne{u}^h_n)}_{\Omega}-{(\text{div}\;\tensorTwo{\eta}^h,bp^h_n)}_{\Omega} 
   ={(\tensorTwo{\eta}^h, \bar{\tensorOne{t}}_{n})}_{\Gamma_\sigma} &&\forall \tensorTwo{\eta}^h \in \boldsymbol{\mathcal{U}}^h_0,\\
  &{(\tensorTwo{\varphi}^h, \mu \tensorTwo{\kappa}^{-1} \cdot \tensorOne{w}^h_n)}_{\Omega} - \sum_{T\in\mathcal{T}_h}\left [{(\text{div} \; 
   \tensorTwo{\varphi}^h, p^h_n)}_{T} 
   - {(\tensorTwo{\varphi}^h \cdot \tensorOne{n}_e, \pi^h_n)}_{\partial T}\right ] = 0 
   &&\forall \tensorTwo{\varphi}^h \in \boldsymbol{\mathcal{W}}^h, \\
  &{(\chi^{h}, b \; \text{div} \; {\tensorOne{u}}^h_n )}_{\Omega} + \Delta t \sum_{T\in\mathcal{T}_h} {(\chi^{h}, \text{div} \; \tensorOne{w}^h_n)}_{T} + {(\chi^h, S_{\epsilon} {p}^h_n)}_{\Omega} = 
   {(\chi^{h},\tilde{f}_n)}_{\Omega} &&\forall \chi^{h} \in \mathcal{P}^h, \\
  &\sum_{T \in \mathcal{T}_h} {-(\zeta^h, \tensorOne{w}_n^h \cdot \tensorOne{n}_e)}_{\partial T}={-(\zeta^h,\bar{q}_n)}_{\Gamma_q} && \forall \zeta^h \in \mathcal{B}_0^h.
\end{align}\label{eq:IBVP_WH}\null
\end{subequations}
\end{linenomath}
Let $\{ \tensorOne{\varphi}_j \}_{j \in \mathcal{N}_w}$ be the $2d \cdot n_T$ basis functions for $\boldsymbol{\mathcal{W}}^h$, where $2d$ is the number of edges (respectively, faces) in a quadrilateral (respectively, hexahedral) element.
Let $\{ \zeta_{\ell} \}_{\ell \in \mathcal{N}_{\pi} \cup \overline{\mathcal{N}}_{\pi}}$ be the basis for $\mathcal{B}^h$, with $\zeta_{\ell}$ the characteristic function of the $\ell$-th edge/face $e_{\ell} \in \mathcal{E}_h$ such that $\zeta_{\ell}(\tensorOne{x}) = 1$ if $\tensorOne{x} \in e_{\ell}$, $\zeta_{\ell}(\tensorOne{x}) = 0$ if $\tensorOne{x} \notin e_{\ell}$.
Sets $\mathcal{N}_{\pi}$ and $\overline{\mathcal{N}}_{\pi}$ identify the indices of basis functions vanishing on $\Gamma_p$ and having support on $\Gamma_p$, respectively.
The same expressions given in \eqref{eq:u_q_p_h} for $\tensorOne{u}^h_n$ and $p^h_n$ are used.
The following approximation for the discontinuous Darcy velocity $\tensorOne{w}^h_n$ and interface pressure $\pi^h_n$ are introduced:
\begin{linenomath}
\begin{align}
  \tensorOne{w}^h_n(\tensorOne{x}) &=
  \sum_{j \in \mathcal{N}_{w} } \tensorOne{\varphi}_j(\tensorOne{x}) w_{j,n}, &
  \pi^h_n(\tensorOne{x}) &=
  \underbrace{\sum_{\ell \in \mathcal{N}_{\pi} } \zeta_\ell(\tensorOne{x}) \pi_{\ell,n}}_{:= \mathring{\pi}^h_n} +
  \underbrace{\sum_{\ell \in \overline{\mathcal{N}}_{\pi} } \zeta_\ell(\tensorOne{x}) \bar{p}_{\ell,n}}_{:= \bar{\pi}^h_n}.
  \label{eq:w_pi_h}
\end{align}
\end{linenomath}
The unknown edge-/face-centered Darcy's velocity components $\{ w_{j,n} \}$ and pressures $\{ \pi_{\ell,n} \}$ at time level $t_n$ are collected in vectors $\Vec{w}_n \in \mathbb{R}^{n_w}$, and $\Vec{\pi}_n \in \mathbb{R}^{n_\pi}$, with $n_w = |\mathcal{N}_w| = 2d \cdot n_T$, and $n_{\pi} = |\mathcal{N}_{\pi}|$.
As in \eqref{eq:u_q_p_h} for $\tensorOne{u}^h_n$ and $\tensorOne{q}^h_n$, the pressure field on the mesh skeleton is expressed as sum of function $\mathring{\pi}^h_n$, which satisfies homogeneous pressure conditions on $\Gamma_p \times \mathcal{I}$, and function $\bar{\pi}^h_n$, which provides a lifting of an approximation of the pressure Dirichlet boundary datum \eqref{massBalanceS_DIR}.
Here, $\{\zeta_\ell\}_{\ell \in \mathcal{N}_{\pi}}$ represents a basis for $\mathcal{B}^h_0$.

Requiring that $\{ \tensorOne{u}^h_n, \tensorOne{w}^h_n, p^h_n, \pi^h_n \}$ given in \eqref{eq:u_q_p_h} and \eqref{eq:w_pi_h} satisfy \eqref{eq:IBVP_WH} for each basis function of $\boldsymbol{\mathcal{U}}^h_0$, $\boldsymbol{\mathcal{W}}^h$, $\mathcal{P}^h$, and $\boldsymbol{\mathcal{B}}^h_0$, produces the following block linear system:
\begin{linenomath}
\begin{align}\label{eq:hybridSys}
  \blkMat{A}_{H} \blkVec{x} &= \blkMat{b}
  \quad
  \text{with}
  \quad
  \blkMat{A}_{H} = 
  \begin{bmatrix}
    A_{uu} &    0            & A_{up} &    0         \\
      0    &  A_{ww}         & A_{wp} & A_{w\pi} \\  
    A_{pu} & \Delta t A_{pw} & A_{pp} &    0         \\
      0    &  A_{\pi w}  &  0     &    0 
  \end{bmatrix}, 
  \quad
  \blkVec{x} =
  \begin{bmatrix}
    \Vec{u}_n \\
    \Vec{w}_n \\
    \Vec{p}_n \\
    \Vec{\pi}_n
  \end{bmatrix},
  \quad
  \blkVec{b} =
  \begin{bmatrix}
    \Vec{f}_u \\
    \Vec{f}_w \\
    \Vec{f}_{p,H} \\
    \Vec{f}_{\pi}
  \end{bmatrix},
\end{align}
\end{linenomath}
with $A_{pw}=-A^{T}_{wp}$, $A_{\pi w}=-A^{T}_{w \pi}$.
The matrix $A_{ww}$ is block diagonal and composed of $n_T$ SPD blocks of size $2d$.
Hence, the block system~\eqref{eq:hybridSys} can be reduced by static condensation, namely
\begin{linenomath}
\begin{align}\label{eq:hmfeSys}
  \blkMat{A}_{H} \blkVec{x} &= \blkMat{b}
  \quad
  \text{with}
  \quad
  \blkMat{A}_{H} = 
  \begin{bmatrix}
    A_{uu} &          A_{up}                    & 0                                      \\  
    A_{pu} &  
    A_{pp} - \Delta t A_{pw} A_{ww}^{-1} A_{wp} & -\Delta t A_{pw} A_{ww}^{-1} A_{w \pi}  \\
    0      & -A_{\pi w} A_{ww}^{-1} A_{wp}      & -A_{\pi w} A_{ww}^{-1} A_{w \pi}     
  \end{bmatrix}, 
  \quad
  \blkVec{x} =
  \begin{bmatrix}
    \Vec{u}_n \\
    \Vec{p}_n \\
    \Vec{\pi}_n
  \end{bmatrix},
  \quad
  \blkVec{b} =
  \begin{bmatrix}
    \Vec{f}_u \\
    \Vec{f}_{p,H} - \Delta t A_{pw } A_{ww}^{-1} \Vec{f}_w \\
    \Vec{f}_{\pi} - A_{\pi w } A_{ww}^{-1} \Vec{f}_w
  \end{bmatrix},
\end{align}
\end{linenomath}
with the final matrix written in a more compact form as:
\begin{linenomath}
\begin{align}
  \blkMat{A}_{H} = 
  \begin{bmatrix}
    A_{uu} &      A_{up}   &          0                   \\  
    A_{pu} & \bar{A}_{pp}  & \Delta t A_{p \pi}       \\
    0      & A_{\pi p} &          A_{\pi \pi}     
  \end{bmatrix}.
  \label{eq:A_33H}
\end{align}
\end{linenomath}
From an implementation point of view, the block matrix~\eqref{eq:A_33H} is constructed directly assembling matrices $\bar{A}_{pp}$, $A_{p \pi}$, and $A_{\pi \pi}$ from element contributions.
Once $\Vec{p}_n$ and $\Vec{\pi}_n$ have been computed, a cell-based reconstruction is used to obtain $\Vec{w}_n$.
Note that $\bar{A}_{pp}$ is a diagonal matrix, while the sparsity patterns of $A_{\pi \pi}$ and $A_{p \pi}$ are the same as $A_{q q}$ and $A_{p q}$, respectively.
The explicit expression for the matrices and right-hand-sides are provided in~\ref{app:mat}.

\section{Stabilized MFE and MHFE Methods}
The selected spaces for the mixed and mixed hybrid formulation can be unstable.
This can occur in the presence of incompressible fluid and solid constituents ($S_{\epsilon} \rightarrow 0$) and  undrained conditions---i.e., $\tensorOne{q}\simeq\tensorOne{0}$ for either low permeability ($\tensorTwo{\kappa} \rightarrow \tensorTwo{0}$) or small time-step size ($\Delta t \rightarrow 0$).
In this situation, the IBVP \eqref{eq:IBVP} degenerates to an undrained steady-state poroelastic problem.
Assuming without loss of generality no fluid source term, both discrete weak form \eqref{eq:IBVP_W} and \eqref{eq:IBVP_WH} become: find  
$\{\tensorOne{u}^h, p^h\} \in \boldsymbol{\mathcal{U}}^h \times \mathcal{P}^h$ such that
\begin{linenomath}
\begin{subequations}
\begin{align}
  &{(\nabla^s \tensorTwo{\eta}^{h}, \tensorFour{C}_{\text{dr}}:\nabla^s \tensorOne{u}^h)}_{\Omega}
   -{(\text{div}\;\tensorTwo{\eta}^{h},b p^h)}_{\Omega} 
   ={(\tensorTwo{\eta}^h, \bar{\tensorOne{t}})}_{\Gamma_\sigma} &&\forall \tensorTwo{\eta}^{h} \in \boldsymbol{\mathcal{U}}^h_0,\\
  &{(\chi^{h}, b \; \text{div} \; {\tensorOne{u}}^h)}_{\Omega} = 0 &&\forall \chi^{h} \in \mathcal{P}^h, \label{eq:IBVP_W_undrained_const}
\end{align}\label{eq:IBVP_W_undrained}\null
\end{subequations}
\end{linenomath}
that is both system \eqref{eq:mfeSys} and \eqref{eq:hmfeSys} reduce to
\begin{linenomath}
\begin{equation}
  \begin{bmatrix}
    A_{uu} & A_{up} \\
    A_{pu} & 0
  \end{bmatrix}
  \begin{bmatrix}
    \Vec{u} \\
    \Vec{p}
  \end{bmatrix}
  =
  \begin{bmatrix}
    \Vec{f}_u \\
    \Vec{0}
  \end{bmatrix}.
  \label{eq:SysSaddlePoint}
\end{equation}
\end{linenomath}
Stability of this saddle point system requires the spaces $\boldsymbol{\mathcal{U}}^h_0$ and $\mathcal{P}^h$ to fulfill the discrete inf-sup condition \cite{BreBat90}, i.e. the following solvability condition must hold true:
\begin{linenomath}
\begin{align}
  &{(\text{div} \; \tensorTwo{\eta}^h, bp^h )}_{\Omega} = 0 \quad \forall \tensorTwo{\eta}^h \in \boldsymbol{\mathcal{U}}^h_0 \quad \Longrightarrow \quad p^h = \text{const}.\label{eq:solv}
\end{align}
\end{linenomath}
Unfortunately, lowest-order continuous finite elements for the displacement field combined with a piecewise-constant interpolation for the pressure do not satisfy \eqref{eq:solv}, hence spurious modes can appear in the pressure solution.
As a remedy, we use a pressure-jump stabilization technique following the approach proposed  in \cite{SilKec90} in the context of Stokes problem.
This technique relies on the construction of macro-elements, on which the discrete solvability condition~\eqref{eq:solv} is satisfied.
In such a case, the macro-element is called stable and it can be proved that the inf-sup condition holds true for any grid constructed by patching together stable macro-elements~\cite{ElmSilWat14}.

\begin{figure}
\centering
\hfill
\begin{subfigure}[b]{.45\linewidth}
\centering
\begin{tikzpicture} [scale=1]
    \def\slice{2.0}
    \def\side{4}

    \draw [line width=0.25mm,color=black] (0,0,\side) -- (\side,0,\side) -- (\side,\side,\side) -- (0,\side,\side) -- (0,0,\side);

    \draw[line width=0.25mm,densely dotted,color=black] (\slice,0,\side) -- (\slice,\side,\side); 
    \draw[line width=0.25mm,densely dotted,color=black] (0,\slice,\side) -- (\side,\slice,\side); 

    \draw[fill=black] (\side,\side,\side) circle (0.2em);
    \draw[fill=black] (\side,0,\side) circle (0.2em);
    \draw[fill=black] (0,0,\side) circle (0.2em);
    \draw[fill=black] (0,\side,\side) circle (0.2em);
    \draw[fill=black] (\slice,\slice,\side) circle (0.2em);

    \draw[fill=black] (\side,\slice,\side) circle (0.2em);
    \draw[fill=black] (\slice,\side,\side) circle (0.2em);
    \draw[fill=black] (\side,\slice,\side) circle (0.2em);
    \draw[fill=black] (\slice,0,\side) circle (0.2em);
    \draw[fill=black] (0,\slice,\side) circle (0.2em);

    \draw [line width=0.25mm,color=black] (6.25*\side/10,-\slice/4,\side) -- (7.5*\side/10,-\slice/4,\side); 
    \draw [line width=0.25mm,densely dotted,color=black] (1.25*\side/10,-\slice/4,\side) -- (2.5*\side/10,-\slice/4,\side); 
    \node at (2*\side/6,-\slice/4,\side) {$\Gamma_{M}$};
    \node at (5*\side/6,-\slice/4,\side) {$\Gamma^{\partial}_{M}$};
\end{tikzpicture}
\caption{}
\end{subfigure}
\hfill
\begin{subfigure}[b]{.45\linewidth}
\centering
\begin{tikzpicture}[scale=.9]
    \def\slice{2.0}
    \def\side{4}

    \filldraw[color=red!50] (0,\slice,0) -- (0,\slice,\side) -- (\side,\slice,\side) -- (\side,\slice,0) -- cycle;
    \filldraw[color=red!50] (0,0,\slice) -- (\side,0,\slice) -- (\side,\side,\slice) -- (0,\side,\slice) -- cycle;
    \filldraw[color=red!50] (\slice,0,0) -- (\slice,\side,0) -- (\slice,\side,\side) -- (\slice,0,\side) -- cycle;

    \draw[line width=0.25mm,densely dotted] (0,\slice,\side) -- (0,\slice,0) -- (\side,\slice,0); 
    \draw (0,\slice,\side) -- (\side,\slice,\side) -- (\side,\slice,0); 

    \draw[line width=0.25mm,densely dotted] (\side,0,\slice) -- (0,0,\slice) -- (0,\side,\slice); 
    \draw (\side,0,\slice) -- (\side,\side,\slice) -- (0,\side,\slice); 

    \draw[line width=0.25mm,densely dotted] (\slice,\side,0) -- (\slice,0,0) -- (\slice,0,\side); 
    \draw (\slice,\side,0) -- (\slice,\side,\side) -- (\slice,0,\side); 

    \draw (\side,0,0) -- (\side,\side,0) -- (0,\side,0);
    \draw (0,0,\side) -- (\side,0,\side) -- (\side,\side,\side) -- (0,\side,\side) -- (0,0,\side);
    \draw (\side,0,0) -- (\side,0,\side);
    \draw (\side,\side,0) -- (\side,\side,\side);
    \draw (0,\side,0) -- (0,\side,\side);
    \draw[line width=0.2mm,densely dotted] (0,0,\side) -- (0,0,0) -- (\side,0,0); 
    \draw[line width=0.2mm,densely dotted] (0,\side,0) -- (0,0,0); 

    \draw[line width=0.2mm,densely dotted,color=black] (\slice,0,\slice) -- (\slice,\side,\slice); 
    \draw[line width=0.2mm,densely dotted,color=black] (0,\slice,\slice) -- (\side,\slice,\slice); 
    \draw[line width=0.2mm,densely dotted,color=black] (\slice,\slice,0) -- (\slice,\slice,\side); 

    \draw[fill=black] (\side,0,0) circle (0.2em);
    \draw[fill=black] (\side,\side,0) circle (0.2em);
    \draw[fill=black] (\side,\side,\side) circle (0.2em);
    \draw[fill=black] (\side,0,\side) circle (0.2em);
    \draw[fill=black] (0,0,\side) circle (0.2em);
    \draw[fill=black] (0,\side,\side) circle (0.2em);
    \draw[fill=black] (0,\side,0) circle (0.2em);

    \draw[fill=black] (\slice,\slice,\slice) circle (0.2em);
    \draw[fill=black] (0,\slice,\slice) circle (0.2em);
    \draw[fill=black] (\slice,0,\slice) circle (0.2em);
    \draw[fill=black] (\slice,\slice,0) circle (0.2em);
    \draw[fill=black] (\side,\slice,\slice) circle (0.2em);
    \draw[fill=black] (\slice,\side,\slice) circle (0.2em);
    \draw[fill=black] (\slice,\slice,\side) circle (0.2em);

    \draw[fill=black] (0,\side,\slice) circle (0.2em);
    \draw[fill=black] (\side,0,\slice) circle (0.2em);
    \draw[fill=black] (\side,\slice,0) circle (0.2em);
    \draw[fill=black] (\side,\slice,\side) circle (0.2em);
    \draw[fill=black] (\slice,\side,\side) circle (0.2em);
    \draw[fill=black] (\side,\slice,\side) circle (0.2em);
    \draw[fill=black] (\side,\side,\slice) circle (0.2em);
    \draw[fill=black] (\slice,0,\side) circle (0.2em);
    \draw[fill=black] (0,\slice,\side) circle (0.2em);
    \draw[fill=black] (0,\slice,0) circle (0.2em);
    \draw[fill=black] (0,0,\slice) circle (0.2em);
    \draw[fill=black] (\slice,0,0) circle (0.2em);
    \draw[fill=black] (\slice,\side,0) circle (0.2em);

    \draw[fill=black] (0,0,0) circle (0.2em);
    \node at (2*\side/6,-\slice/4,\side) {$\Gamma_{M}$};
    \node at (5*\side/6,-\slice/4,\side) {$\Gamma^{\partial}_{M}$};
    \draw[black, line width=0.075em, fill=white] (3.85*\side/6,-1*\side/6,\side) rectangle (4.35*\side/6,-0.5*\side/6,\side);
    \draw[black, line width=0.075em, fill=red!50] (0.85*\side/6,-1*\side/6,\side) rectangle (1.35*\side/6,-0.5*\side/6,\side);
\end{tikzpicture}
\caption{3D}
\end{subfigure}
\hfill\null
\caption{ Reference macroelement patch in 2D (a) and 3D (b). Edges and faces subject to the jump stabilization are indicated with a dotted line and in red, respectively.}\label{fig:macroel}
\end{figure}
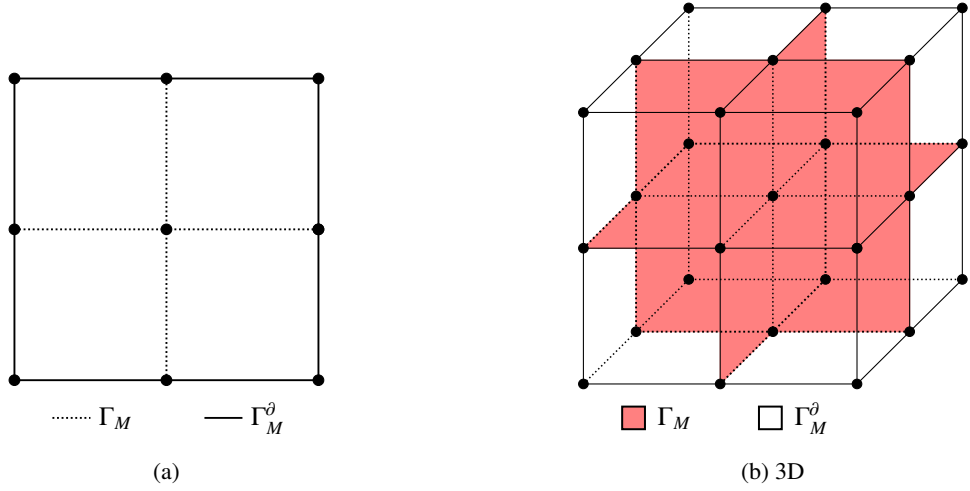

Let $\mathcal{M}_h$ be the set of macroelements, i.e. the union of four quadrilaterals in 2D and eight hexahedra in 3D.
We denote by $\Gamma^{\partial}_M$  and $\Gamma_M$  the external and internal boundary of $M$, respectively, for a macroelement $M$ (Fig.~\ref{fig:macroel}). 
The stabilization consists of relaxing the incompressibility constraint by adding the term $J(\chi^h,p^h)$ so that \eqref{eq:IBVP_W_undrained_const} becomes:
\begin{linenomath}
\begin{align}
  &{(\chi^h, b \; \text{div} \; {\tensorOne{u}}^h)}_{\Omega} + J(\chi^h, p^h) 
  = 0 &&\forall \chi^h \in \mathcal{P}^h,
  \label{eq:STporoS_W}
\end{align}
\end{linenomath}
where 
\begin{linenomath}
\begin{align}
  J(p^h,\chi^h)&= \sum_{M\in\mathcal{M}_h} \beta_M |M| \sum_{e \in \Gamma_M} \llbracket \chi^h \rrbracket_e \llbracket p^h \rrbracket_e.
  \label{eq:J_def}
\end{align}
\end{linenomath}
In~\eqref{eq:J_def}, ${\llbracket \cdot \rrbracket}_e$ denotes the jump across the edge/face $e$, $|M|$ is the $d$-measure of $M$, and $\beta_M$ is a stabilization term depending on the physical parameters.
It is worth noticing that, unlike other stabilization techniques~\cite{Rod_etal18,NiuRuiHu19}, the pressure-jump method has also a physical interpretation.
Indeed, $J(p^h,\chi^h)$ can be regarded as a fictitious flux introduced through the inner edges or faces of each macro-element.
The local fictitious flux along $e$, which is proportional to $\beta_M$ and the jump of $p^h$, is introduced so as to compensate the spurious fluxes induced by non-physical pressure oscillations across adjacent elements.
Element-wise mass-conservation no longer holds, but is guaranteed on the macro-element, because only the jumps across the inner edges or faces are considered.
Note also that this fictitious flux is effective in undrained conditions only, becoming irrelevant in drained configurations where physical fluxes prevail.

The stabilization parameter $\beta_M$ at the macro-element level is user-specified.
As proposed in~\cite{ElmSilWat14}, an optimal candidate depends on the eigenspectrum of the local Schur complement matrix $B^M_p$ computed for the macro-element $M$:
\begin{linenomath}
\begin{align}
  B^M_p= A^M_{stab} - A^{M}_{pu} {\left[A^{{M}}_{uu}\right ]}^{-1} A^{M}_{up},
\end{align}
\end{linenomath}
where $A^M_{uu}$, $A^M_{up}$, $A^M_{pu}$ are the blocks introduced in~\eqref{eq:mfeSys} and~\eqref{eq:hybridSys} restricted to $M$, with homogeneous Dirichlet conditions for the displacements on $\Gamma^{\partial}_M$, and $A^M_{stab}$ is the matrix form of $J$ in $M$.
The key idea is to set $\beta_M$ such that the non-zero extreme eigenvalues of $B^M_p$ are not affected by the introduction of the stabilization contribution $A^M_{stab}$.
Following the analysis in~\cite{ElmSilWat14} for 2D problems, we can set $\beta_M={(b/2)}^2/(2 G +\lambda)$, where $G$ and $\lambda$ are the Lam\'e parameters on the macro-element $M$.
For 3D problems, a recent analysis has been carried out for a mixed finite element-finite volume formulation of multiphase poromechanis \cite{CamWhiBor19}.
Extending those results, we can set $\beta_M={(3 b)}^2 / [32 (\lambda + 4 G)]$.

The introduction of the stabilization terms adds the matrix $A_{stab}$ obtained from assembling $A^M_{stab}$ to the diagonal block of the discrete balance equations, yielding:
\begin{linenomath}
\begin{align}
    \blkMat{A}_{M} &= 
  \begin{bmatrix}
    \Mat{A}\sub{uu} &      0          & \Mat{A}\sub{up} \\  
         0          & \Mat{A}\sub{qq} & \Mat{A}\sub{qp} \\
    \Mat{A}\sub{pu} & \Delta t \Mat{A}\sub{pq} & \Mat{A}\sub{pp}+A_{stab}
  \end{bmatrix}, 
  &
  \blkMat{A}_{H} &= 
  \begin{bmatrix}
    A_{uu} &      A_{up}   &          0                   \\  
    A_{pu} & \bar{A}_{pp}+A_{stab}  & \Delta t A_{p \pi}       \\
    0      & A_{\pi p} &          A_{\pi \pi}     
  \end{bmatrix}.
  \label{eq:stab_A}
\end{align}
\end{linenomath}
Recalling that the pattern of the blocks $A_{qp}$ and $A_{\pi p}$ is the same---providing the face-to-element connectivity, i.e. ${[A_{\pi p}]}_{ij},{[A_{q p}]}_{ij}\neq 0$ if face $i$ belongs to element $j$---it is easy to see that the sparsity pattern of $A_{stab}$ is a subset of the sparsity pattern of $A_{pq}A_{qp}$ and $A_{p \pi} A_{\pi p}$.
Notice that the matrices $\blkMat{A}_{M}$ and $\blkMat{A}_{H}$ in~\eqref{eq:stab_A} are non-symmetric. Although they could be easily symmetrized,
in this work we prefer keeping the non-symmetry because the symmetric form would be indefinite anyway.
The topic was widely investigated for instance by Benzi et al.~\cite{BenGolLie05} for saddle-point problems, showing that the performance difference between symmetrized and non-symmetrized formulations is usually marginal.

\section{Linear solver}\label{sec:ls}
The efficient solution of the linear systems with the non-symmetric block matrices~\eqref{eq:stab_A} by Krylov subspace methods requires the development of dedicated preconditioning strategies.
A robust and effective family of preconditioners for MFE poromechanics is provided by block triangular preconditioners based on a Schur complement-approximation strategy, e.g.,~\cite{CasWhiFer16}. 
In this work, we follow a similar approach for the stabilized MHFE problem \eqref{eq:stab_A} by defining the block upper triangular factor: 
\begin{linenomath}
\begin{align}\label{eq:hybridPrec}
\blkMat{M}=
\begin{bmatrix}
A_{uu} & A_{up}& 0 \\
0 & \tilde{B}_{p}& \Delta t A_{p\pi}  \\
0 &  0 & \tilde{C}_{\pi}  
\end{bmatrix} 
\end{align}
\end{linenomath}
where $\tilde{B}_p$ is an approximation of the first-level Schur complement $B_p=\bar{A}_{pp} +A_{stab}- A_{pu} A^{-1}_{uu} A_{up}$, and $\tilde{C}_{\pi}$  is the preconditioner second-level Schur complement $A_{\pi \pi} - \Delta t A_{\pi p} \tilde{B}^{-1}_{p} A_{p \pi}$.

The following results provide information on: (i) the eigenspectrum of the right-preconditioned matrix $\blkMat{A}_{H}\blkMat{M}^{-1}$ for any choice of $\tilde{B}_p$, and (ii) the regularity of $\tilde{C}_\pi$, whose inverse is required to apply $\blkMat{M}^{-1}$.

\begin{thm}\label{th:eigen}
Let $\blkMat{A}_H$ and $\blkMat{M}$ be the matrices introduced in~\eqref{eq:stab_A} and~\eqref{eq:hybridPrec}, respectively. 
Then, the eigenvalues of $\blkMat{T}=\blkMat{A}_H \blkMat{M}^{-1}$ are either 1, with multiplicity $n_u + n_{\pi} - n_p$, or equal to $1 + \mu_i$, where $\mu_i$ are the nonzero eigenvalues of $\blkMat{Z}$:
\begin{linenomath}
\begin{align}
\blkMat{Z} &= 
\begin{bmatrix}
E_p & - \Delta t E_p A_{p \pi} \tilde{C}^{-1}_{\pi}  \\
A_{\pi p} \tilde{B}^{-1}_p & 0
\end{bmatrix},
\label{eq:Z_def}
\end{align}
\end{linenomath}
with $E_p = B_p \tilde{B}^{-1}_p - I_p$.
\end{thm}

\begin{proof}
Recalling that the inverse of $\blkMat{M}$ reads:
\begin{linenomath}
\begin{align}
\blkMat{M}^{-1} = 
\begin{bmatrix}
  A^{-1}_{uu} & - A^{-1}_{uu} A_{up} \tilde{B}^{-1}_p & \Delta t A^{-1}_{uu} A_{up} \tilde{B}^{-1}_p A_{p \pi} \tilde{C}^{-1}_{\pi} \\
   0          &  \tilde{B}_p^{-1}                     & - \Delta t \tilde{B}^{-1}_p A_{p\pi} \tilde{C}_{\pi}^{-1}                   \\
   0          &      0                                & \tilde{C}^{-1}_{\pi}
\end{bmatrix},
\end{align}
\end{linenomath}
the matrix $\blkMat{T}$ is:
\begin{linenomath}
\begin{align}\label{eq:matT}
\blkMat{T} = \blkMat{A}_H \blkMat{M}^{-1} &= 
\begin{bmatrix}
    I_u              &       0                          &                 0                             \\
  A_{pu} A_{uu}^{-1} &   B_p \tilde{B}_p^{-1}           & - \Delta t E_p A_{p \pi } \tilde{C}^{-1}_{\pi} \\
    0                &   A_{\pi p} \tilde{B}^{-1}_p &                 I_{\pi}
\end{bmatrix}
= \blkMat{I} +
\begin{bmatrix}
    0                &   0                             &   0                                           \\
  A_{pu} A_{uu}^{-1} &  E_p                            & - \Delta t E_p A_{p \pi } \tilde{C}^{-1}_{\pi} \\
    0                &  A_{\pi p} \tilde{B}^{-1}_p &   0
\end{bmatrix}.
\end{align}
\end{linenomath}
From equation~\eqref{eq:matT} it follows that the eigenvalues of $\blkMat{T}$ are 1 with multiplicity $n_u$ and the other $n_\pi + n_p$ are equal to those of $\blkMat{I} + \blkMat{Z}$.
However, $\blkMat{Z}$ has at most rank $2n_p$, since $n_\pi>n_p$ and $\ker(A_{p \pi })$ has at least dimension $n_{\pi} - n_p$.
\end{proof}

\begin{lem}\label{lm:bound}
The non-zero eigenvalues of $\blkMat{Z}$ in~\eqref{eq:Z_def} satisfy the upper bound:
\begin{linenomath}
\begin{equation}
    \left|\mu_i\right| \leq \varepsilon + \sqrt{\varepsilon^2 + 2 \Delta t \gamma \varepsilon},
    \label{eq:eigen_bound}
\end{equation}
\end{linenomath}
where $\varepsilon=\|E_p\|/2$ and $\gamma=\|A_{p\pi}\tilde{C}_\pi^{-1}A_{\pi p}\tilde{B}_p^{-1}\|$ for any compatible matrix norm.
\end{lem}

\begin{proof}
Let $\mu_i$ be a non-zero eigenvalue of $\blkMat{Z}$, with $\vec{v}=\begin{bmatrix} \vec{v}_p \\ \vec{v}_\pi \end{bmatrix}$ the corresponding eigenvector:
\begin{linenomath}
\begin{equation}
    \left\{ \begin{array}{l}
    E_p \vec{v}_p - \Delta t E_p A_{p\pi} \tilde{C}_\pi^{-1} \vec{v}_\pi = \mu_i \vec{v}_p \\
    A_{\pi p} \tilde{B}_p^{-1} \vec{v}_p = \mu_i \vec{v}_\pi
    \end{array} \right..\label{eq:eigen_exp}
\end{equation}
\end{linenomath}
Eliminating $\vec{v}_\pi$ from equation~\eqref{eq:eigen_exp} and taking compatible norms we obtain:
\begin{linenomath}
\begin{equation}
    \left|\mu_i\right| \left\| \vec{v}_p \right\| \leq 2 \varepsilon \left\| \vec{v}_p \right\| + 2 \frac{\Delta t}{\left|\mu_i\right|} \gamma \varepsilon \left\| \vec{v}_p \right\|.
    \label{eq:eigen_norm}
\end{equation}
\end{linenomath}
Since $\vec{v}_p\neq\vec{0}$,~\eqref{eq:eigen_norm} is equivalent to the inequality:
\begin{linenomath}
\begin{equation}
    \left| \mu_i \right|^2 - 2 \varepsilon \left| \mu_i \right| - 2 \Delta t \gamma \varepsilon \leq 0,
    \label{eq:eigen_ineq}
\end{equation}
\end{linenomath}
which yields the bound~\eqref{eq:eigen_bound}.
\end{proof}

\begin{rem}
Theorem~\ref{th:eigen} and the related Lemma~\ref{lm:bound} show that the quality of the approximation $\tilde{B}_{p}$ controls at most $2n_p$ eigenvalues of $\blkMat{T}$. 
If the quality of $\tilde{B}_{p}$ improves, so does the clustering of the eigenspectrum of $\blkMat{T}$.
In the limit of $\|E_p\|=0$, all the eigenvalues of $\blkMat{T}$ are unitary independently of $\Delta t$, which plays a secondary role.
\end{rem}

\begin{rem}
We recall that a clustered eigenspectrum far from 0 is not a sufficient condition to ensure
a fast GMRES or Bi-CGStab convergence,
because the solver behavior depends also on the eigenvectors of the preconditioned matrix.
It is well-known that we can build matrices with all unitary eigenvalues whose convergence can be achieved by GMRES only after a number of iterations on the order of the system size~\cite{GrePtaStr96}.
However, a wide computational experience with matrices arising from discretized PDEs shows that a preconditioned non-symmetric matrix with a compact eigenspectrum far from 0 very rarely yields
poor convergence. Hence, even though Theorem~\ref{th:eigen} and Lemma~\ref{lm:bound} do not provide a rigorous
convergence result for solvers like GMRES or Bi-CGStab, this outcome suggests that the proposed preconditioner is expected to be rather effective.
\end{rem}

\begin{thm}\label{th:C_spd}
Let $C_{\pi}$ be the second level Schur complement 
of $\blkMat{A}_{H}$ in~\eqref{eq:stab_A}:
\begin{linenomath}
\begin{align}
C_\pi = A_{\pi \pi} - \Delta t A_{\pi p} B_p^{-1} A_{p\pi}.
\label{eq:Cpi}
\end{align}
\end{linenomath}
Then, $C_{\pi}$ is symmetric positive definite.
\end{thm}

\begin{proof}
The symmetry of $C_\pi$ follows immediately by construction.
The proof of the positive definiteness can be carried out according to the procedure sketched in~\cite{ChaJaf86,HuiKaa92}.
Using the block definitions previously introduced, $C_\pi$ reads:
\begin{linenomath}
\begin{equation}
    C_{\pi} = -A_{\pi w} A_{ww}^{-1} A_{w \pi} - \Delta t A_{\pi w} A_{ww}^{-1} A_{w p} {\left( A_{pp} + A_{stab} - \Delta t A_{pw} A_{ww}^{-1} A_{wp} - A_{pu} A_{uu}^{-1} A_{up} \right)}^{-1} A_{pw}  A_{ww}^{-1} A_{w \pi}.
    \label{eq:Cpi_full}
\end{equation}
\end{linenomath}
Let $\Vec{\pi}$ be a non-null vector in $\mathbb{R}^{n_{\pi}}$. Recalling that $A_{\pi w}=-A_{w\pi}^T$, we have:
\begin{linenomath}
\begin{align}
  \Vec{\pi}^T C_{\pi} \Vec{\pi}
  &= {(A_{w \pi }\Vec{\pi})}^T A_{ww}^{-1} (A_{w \pi} \Vec{\pi})
   + \Delta t {(A_{w \pi } \Vec{\pi})}^T A_{ww}^{-1} A_{wp} 
   {( H - \Delta t A_{pw} A_{ww}^{-1} A_{wp})}^{-1} A_{pw} A_{ww}^{-1} (A_{w \pi} \Vec{\pi}),
  \label{eq:S2_SPD_1}
\end{align}
\end{linenomath}
with $H = A_{pp} + A_{stab} - A_{pu} A_{uu}^{-1} A_{up}$.
Defining $\vec{p} \in \mathbb{R}^{n_p}$ as:
\begin{linenomath}
\begin{align}
  \vec{p} = \Delta t {( H - \Delta t  A_{pw} A_{ww}^{-1} A_{wp})}^{-1} A_{pw} A_{ww}^{-1} (A_{w\pi} \Vec{\pi}),
  \label{eq:S2_SPD_2}
\end{align}  
\end{linenomath}
equation~\eqref{eq:S2_SPD_1} becomes:
\begin{linenomath}
\begin{align}
  \Vec{\pi}^T C_{\pi} \Vec{\pi}
  &= {(A_{w \pi}\Vec{\pi})}^T A_{ww}^{-1} (A_{w \pi} \Vec{\pi})
   + {(A_{w \pi} \Vec{\pi})}^T A_{ww}^{-1} (A_{wp} \Vec{p}).
  \label{eq:S2_SPD_3}
\end{align}
\end{linenomath}
From the definition~\eqref{eq:S2_SPD_2}, it follows:
\begin{linenomath}
\begin{align}
  \frac{1}{\Delta t} \Vec{p}^T H \Vec{p} + {(A_{wp} \Vec{p})}^T A_{ww}^{-1} (A_{wp} \Vec{p} + A_{w\pi} \Vec{\pi}) = 0,
  \label{eq:S2_SPD4}
\end{align}
\end{linenomath}
which can be added to the right-hand side of~\eqref{eq:S2_SPD_3}:
\begin{linenomath}
\begin{align}
  \Vec{\pi}^T C_{\pi} \Vec{\pi}
  &= {(A_{w\pi}\Vec{\pi})}^T A_{ww}^{-1} (A_{w\pi} \Vec{\pi})
   + {(A_{w\pi} \Vec{\pi})}^T A_{ww}^{-1} (A_{wp} \Vec{p})
   + \frac{1}{\Delta t} \Vec{p}^T H \Vec{p}
   + {(A_{wp} \Vec{p})}^T A_{ww}^{-1} (A_{wp}\Vec{p} + A_{w \pi }\Vec{\pi}) \nonumber \\
  &= \frac{1}{\Delta t} \Vec{p}^T H \Vec{p}
   + {(A_{wp} \Vec{p} + A_{w\pi} \Vec{\pi})}^T A_{ww}^{-1} (A_{wp}\Vec{p} + A_{w\pi}\Vec{\pi}).
  \label{eq:S2_SPD_5}
\end{align}
\end{linenomath}
Recall that $H$ is SPD because $A_{pu}=-A_{up}^T$. Hence, $C_\pi$ is also positive definite because
$\Vec{\pi}^T C_{\pi} \Vec{\pi}$ is the sum of quadratic forms of SPD matrices.
\end{proof}

In the preconditioner~\eqref{eq:hybridPrec}, we use an approximation $\tilde{C}_\pi$ of $C_\pi$, which is obtained by replacing the Schur complement $B_p$ with $\tilde{B}_p$.
In this case, we define $\tilde{B}_p$ as:
\begin{linenomath}
\begin{equation}
    \tilde{B}_p = \tilde{H} - \Delta t A_{pw} A_{ww}^{-1} A_{wp},
    \label{eq:app_Bp}
\end{equation}
\end{linenomath}
where $\tilde{H}$ is a diagonal matrix with positive entries replacing the matrix $H$ introduced in the proof of Theorem~\ref{th:C_spd}.
Because $A_{pw}A_{ww}^{-1}A_{wp}$ is diagonal as well, $\tilde{B}_p$ can be inverted straightforwardly.

\begin{rem}
The proof of Theorem~\ref{th:C_spd} requires only $B_p=H-\Delta t A_{pw}A_{ww}^{-1}A_{wp}$ for a positive definite matrix $H$. Hence, replacing $H$ with $\tilde{H}$ as defined above and $B_p$ with $\tilde{B}_p$ of equation~\eqref{eq:app_Bp} allows to conclude that also $\tilde{C}_\pi$ is guaranteed to be SPD\@.
\end{rem}

\begin{lem}\label{thm:lam_C}
Any eigenvalue $\lambda$ of $\tilde{C}_{\pi}$ with $\tilde{B}_{p}$ as defined in~\eqref{eq:app_Bp} reads:
\begin{linenomath}
\begin{equation}
  \lambda = a - \sum^{n_p}_{j=1} \frac{b^2_j \Delta t}{c_{j} + d_{j} \Delta t }
\end{equation}
\end{linenomath}
where $a$, $\vec{b}={[b_1,\ldots,b_{n_p}]}^T$, $\vec{c}={[c_1,\ldots,c_{n_p}]}^T$ and $\vec{d}={[d_1,\ldots,d_{n_p}]}^T$ are:
\begin{linenomath}
\begin{align}
 a&= \vec{v}^{T} A_{\pi \pi} \vec{v}, \nonumber &
 \vec{b}&= A_{p \pi} \vec{v}, \nonumber \\
 \vec{c}&= \operatorname{diag}({\tilde{H}} ), \nonumber &
 \vec{d}&= \operatorname{diag}(A_{wp}^T A^{-1}_{ww} A_{wp}), \nonumber 
\end{align}
\end{linenomath}
with $\vec{v}$ the eigenvector associated to $\lambda$ such that $\|\vec{v}\|_2=1$.
\end{lem}
\begin{proof}
The lemma follows immediately from writing the Rayleigh quotient of $\tilde{C}_\pi$.
Since $\tilde{C}_\pi$ is SPD, $\lambda$ reads: 
\begin{linenomath}
\begin{equation}
  \lambda = \vec{v}^{T} \tilde{C}_{\pi} \vec{v} = 
   \vec{v}^{T} A_{\pi \pi} \vec{v} - \Delta t  \vec{v}^{T} A_{\pi p} \tilde{B}^{-1}_{p} A_{p \pi} \vec{v}.
\end{equation}
\end{linenomath}
Recalling that $\tilde{B}_p$ is diagonal, 
we have:
\begin{linenomath}
\begin{align}
  \lambda &= a - \Delta t \vec{b}^T {(\tilde{H} + \Delta t A_{wp}^T A_{ww}^{-1} A_{wp} )}^{-1} \vec{b} \nonumber \\
          &= a - \sum^{n_p}_{j=1} \frac{b^2_j \Delta t}{ c_{j} + d_j\Delta t }.  
\end{align}
\end{linenomath}
\end{proof}

\begin{rem}
Notice that $a$, $c_j$ and $d_j$ are strictly positive numbers for any $j=1,\ldots,n_p$. Hence,
Lemma~\ref{thm:lam_C} shows that
any eigenvalue of $\tilde{C}_\pi$ 
is a positive function 
monotonically decreasing with $\Delta t\in\left (0,+\infty \right)$ and bounded between $a$ and $a-\vec{b}^T {(A_{wp}^T A^{-1}_{ww} A_{wp})}^{-1}\vec{b}$. 
Therefore, the condition number of $\tilde{C}_\pi$ is also bounded for any $\Delta t$.
\end{rem}

\section{Numerical results}
Three sets of numerical experiments are used to investigate both the accuracy of the stabilization technique and the computational efficiency of the preconditioner.
The first set (Test 1) arises from Barry-Mercer's problem~\cite{BerMer99}, i.e., a 2D benchmark of linear poroelasticity.
This problem is used to verify the theoretical error convergence and test the accuracy of the stabilization.
The second set is an impermeable cantilever beam (Test 2), which is used to study the stabilization effectiveness and to perform an analysis of the preconditioner weak scalability. 
Finally, we consider a field application (Test  3) to investigate the preconditioner robustness with respect to a strong variability of the governing material parameters, computational efficiency and strong scalability in a parallel context.

In all test cases, GMRES~\cite{SaaSch86} with right preconditioning is selected as the Krylov subspace method with zero initial guess.
Iterations are stopped when the 2-norm of the initial residual is reduced below a user-specified tolerance $\tau=10^{-6}$.
The computational performance is evaluated in terms of number of iterations $n_{it}$, 
CPU time in seconds for the preconditioner construction $T_p$ and for the Krylov solver to converge $T_s$.
The total time is denoted by $T_t = T_p + T_s$.
All computations are performed: 
(a) on an Intel Core i7 4770 processor at 3.4 GHz with 8-GB of memory for serial simulations; (b) on a high performance cluster with nodes containing two Intel Xeon E5-2695 18-core processors sharing 128 GiB of memory on each node with Intel Omni-Path interconnects
between nodes for parallel simulations.
These numerical experiments have been implemented using \textit{Geocentric}, a simulation framework for computational geomechanics \cite{WhiBor11} that relies heavily on finite element infrastructure from the deal.ii library~\cite{dealII91}. 
The PETSc suite~\cite{petsc-web-page} is used as a linear algebra package.

We consider the following variants of the proposed preconditioner:
\begin{linenomath}
\begin{align}
\blkMat{M}^{(h)}_{I}&=\blkMat{M}(A_{uu},\tilde{B}_{p},\tilde{C}_{\pi})\\ 
\blkMat{M}^{(h)}_{II}&=\blkMat{M}(A^{(\text{gamg\_s})}_{uu},\tilde{B}_{p},\tilde{C}^{(\text{B\_amg})}_{\pi}) \\
\blkMat{M}^{(h)}_{III}&=\blkMat{M}(A^{(\text{gamg\_r})}_{uu},\tilde{B}_{p},\tilde{C}^{(\text{B\_amg})}_{\pi}) \label{eq:varhprec}
\end{align}
\end{linenomath}
In $\blkMat{M}^{(h)}_{I}$ we implement the exact application of $A^{-1}_{uu}$, $\tilde{B}^{-1}_{p}$, $\tilde{C}^{-1}_{\pi}$ by a nested direct solver.
The only approximation here is the substitution of the exact Schur complement $B_p$ with the diagonal approximation $\tilde{B}_p=A_{pp}+\text{diag}(A_{stab})+\text{diag}(A_{up}^T \text{diag}(A_{uu})^{-1} A_{up})$. 
In contrast, approaches based on $\blkMat{M}^{(h)}_{II}$ and $\blkMat{M}^{(h)}_{III}$  introduce further levels of approximation by utilizing algebraic multigrid preconditioners for each sub-problem.
The algebraic multigrid selected for $A^{-1}_{uu}$ is GAMG~\cite{petsc-user-ref} with smoothed aggregation, using either the Separate Displacement Component (SDC) approach (superscripts ``gamg\_s'' in $\blkMat{M}^{(h)}_{II}$), or the near kernel information provided by the Rigid Body Modes (RBM) (superscripts ``gamg\_r'' in $\blkMat{M}^{(h)}_{III}$).
In both $\blkMat{M}^{(h)}_{II}$ and $\blkMat{M}^{(h)}_{III}$, the classical AMG method \cite{RugStu87} as implemented in the Hypre package \cite{Falgout02} is used for $\tilde{C}_{\pi}$.

We compare this preconditioner with the Block Triangular preconditioner (BTP) originally developed in~\cite{CasWhiFer16} for the mixed system~\eqref{eq:mfeSys}:
\begin{linenomath}
\begin{align}
\blkMat{M}(M_{A_{uu}},M_{A_{qq}},M_{C_{p}}) = 
\begin{bmatrix}
M_{A_{uu}} & 0  & 0 \\  
0  & M_{A_{qq}} & 0 \\
\Mat{A}_{pu} & \Delta t \Mat{A}_{pq} & M_{C_{p}}
\end{bmatrix}, \label{eq:prec}
\end{align}
\end{linenomath}
with $\Mat{C}_p = A_{stab} + A_{pp} - A_{pu} A^{-1}_{uu} A_{up} - \Delta t A_{pq} A^{-1}_{qq} A_{q p}$, and $M_{A_{uu}}$, $M_{A_{q q}}$ and $M_{C_p}$ inner preconditioners for $A_{uu}$, $A_{q q}$ and $C_{p}$, respectively.
The Schur complement $C_p$ is replaced by $\tilde{C}_p$, where the contribution $A_{pu} A^{-1}_{uu} A_{up}$ is approximated by the diagonal fixed-stress matrix \cite{CasWhiTch15,WhiCasTch16} and $A_{p q} {A}^{-1}_{q q} A_{q p }$ by replacing $A_{q q }$ with a lumped spectrally equivalent matrix~\cite{BerManMan98}.
Note that, since $A_{pp}$ is diagonal, the sparsity pattern of $\tilde{C}_{p}$ is still that of $A_{p q} A_{q p}$ independently of the presence of $A_{stab}$ \cite{FerFriCasWhi20}.
Following the notation used in~\eqref{eq:varhprec}, we define two variants for BTP: 
\begin{linenomath}
\begin{align}
\blkMat{M}^{(m)}_{I}&=\blkMat{M}(A_{uu},{A}_{q q},\tilde{C}_{p})\\ 
\blkMat{M}^{(m)}_{II}&=\blkMat{M}(A^{(\text{gamg\_s})}_{uu},{A}^{(\text{ic})}_{q q},\tilde{C}^{(\text{B\_amg})}_{p})
\end{align}
\end{linenomath}
Note that the use of an incomplete Cholesky factorization with zero prescribed degree of fill-in for the matrix $A_{q q}$ in the $\blkMat{M}^{(m)}_{II}$ variant is sufficient to obtain optimal performances~\cite{CasWhiFer16}.

\subsection{Barry-Mercer's problem}

\begin{figure}
\centering
  \hfill
  \begin{subfigure}[b]{.45\linewidth}
  \centering
  \begin{tikzpicture}[>=latex,font=\small]
    \draw [thick] plot coordinates {(0,3.0) (3.0,3.0)};
    \draw [thick] plot coordinates {(3.0,3.0) (3.0,0.0)};
    \draw [thick] plot coordinates {(3.0,0.0) (0.0,0.0)};
    \draw [thick] plot coordinates {(0.0,0.0) (0.0,3.0)};
    \draw [->] (-1.7,0) -- (-1.7,1.);
    \draw [->] (-1.7,0) -- (-0.7,0);
    \draw [thin] plot coordinates {(-0.03,3.0) (0.03,3.0)};
    \draw [thin] plot coordinates {(3.0,-0.03) (3.0,0.03)};
    \node [left] at (-1.7,1.) {$y$};
    \node [below] at (-0.7,0) {$x$};
    \node at (1.0,0.8) { 
    \small {$f$}};
    \node [left] at (0,3) { 
    \small {$(0,l)$}};
    \node [below] at (3.,0) { 
    \small{$(l,0)$}};
    \node [below] at (0.,0) { 
    \small{$(0,0)$}};
    \fill (0.75,0.75) circle (0.2em);
    \fill (3,0) circle (0.2em);
    \fill (0,3) circle (0.2em);
    \fill (0,0) circle (0.2em);
    \node at (-0.5,1.5)  [rotate=90] { 
    \small{$p=u_y=\frac{\displaystyle \partial u_x}{\displaystyle \partial x} =0$}};
    \node at (3.5,1.5)  [rotate=90] { 
    \small{$p=u_y=\frac{\displaystyle \partial u_x}{\displaystyle \partial x} =0$}};
    \node at (1.5,-0.5) { 
    \small{$p=u_x=\frac{\displaystyle \partial u_y}{\displaystyle \partial y} =0$}};
    \node at (1.5,3.5) { 
    \small{$p=u_x=\frac{\displaystyle \partial u_y}{\displaystyle \partial y} =0$}};
    \end{tikzpicture}
    \caption{}
  \end{subfigure}
  \hfill
  \begin{subfigure}[b]{.45\linewidth}
    \small
    \centering
    \begingroup
    \renewcommand{\arraystretch}{1.4} 
    \begin{tabular}{lll}
      \toprule
      Quantity                          & Value & Unit \\
      \midrule   
      Young's modulus ($E$)             & $1 \times 10^5$    & [Pa] \\  
      Poisson's ratio ($\nu$)           & $0.1$              & [-] \\ 
      Biot's coefficient ($b$)          & $1.0$              & [-] \\ 
      \begin{tabular}[c]{@{}Sl@{}}
        Constrained specific \\[-7pt]
        storage ($S_{\epsilon}$)
      \end{tabular}                     & $0$                & [Pa] \\
      Isotropic permeability ($\kappa$) & $1 \times 10^{-9}$ & [m$^2$] \\
      Fluid viscosity ($\mu$)           & $1 \times 10^{-3}$ & [Pa $\cdot$ s] \\
      Domain size $x$-$y$ ($l$)         & $1.0$              & [m] \\
      \bottomrule
    \end{tabular}
    \endgroup
    \caption{}
  \end{subfigure}
  \hfill\null
\caption{Barry-Mercer's problem: (a) domain sketch and (b) physical parameters.}\label{fig:BM_dom}
\end{figure}

An analytical validation test for poroelasticity is Barry-Mercer's problem~\cite{BerMer99}, which describes flow and deformation due to a point-source sine wave on a square domain $[0,l] \times [0,l]$ (Fig.~\ref{fig:BM_dom}). 
The periodic point source term is located at $\Vec{x}_0$ and is given by:
\begin{linenomath}
\begin{equation}
f(t) = 2 \hat{\beta} \delta (\Vec{x} - \Vec{x}_0) \sin (\hat{\beta} t)
\end{equation} 
\end{linenomath}
with $\delta(\cdot)$ the Dirac function and $\hat{\beta}=(\lambda + 2 G) \kappa / \mu$. 
All sides are constrained with zero pressure and zero tangential displacement boundary conditions.
The hydro-mechanical properties are provided in Fig.~\ref{fig:BM_dom}.

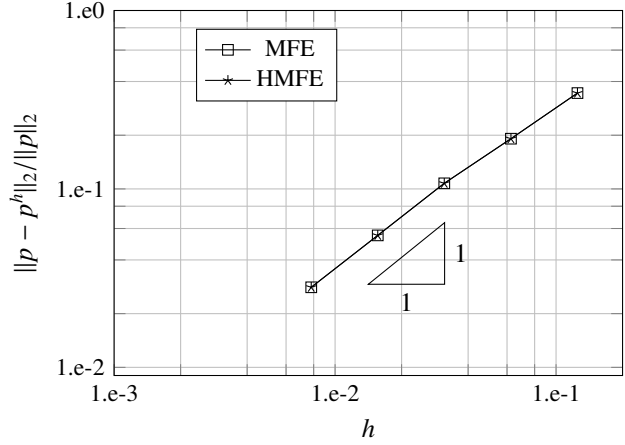
\begin{figure}
  \centering
  \begin{tikzpicture}
    \begin{loglogaxis}[ scale=0.45,width=\linewidth,height=.75\linewidth, grid=major,xmin=1.e-3,xmax=2.e-1,ymin=9.e-3,ymax=1.e-0,
      xlabel={ $h$ }, ylabel={ $\|p-p^h\|_2/\|p\|_2$  },
      xtick={1.e-3,2.e-3,4.e-3,6.e-3,8.e-3,1.e-2,2.e-2,4.e-2,6.e-2,8.e-2,1.e-1},
      xticklabels={{\small 1.e-3},{\small },{\small },{\footnotesize },{\small },
                   {\small 1.e-2 },{\small },{\small },{\footnotesize },{\small },{\small 1.e-1}},
      ytick={1.e-2,2.e-2,4.e-2,6.e-2,8.e-2,1.e-1,2.e-1,4.e-1,6.e-1,8.e-1,1.e0},
      yticklabels={{\small1.e-2},{\small },{\small },{\small },{\small },
                  {\small 1.e-1},{\small },{\small },{\small },{\small },{\small 1.e0}},
      ylabel near ticks,xlabel near ticks,
      legend style={at={(0.3,0.95)}, anchor=north }
]
      \addplot [black,mark=square] table [x=h,y=err] {./error_BM.txt};
      \addplot [black,mark=star] table [x=h,y=err] {./error_BM.txt};
      \logLogSlopeTriangle{0.65}{0.15}{0.25}{1}{black};
      \legend{\small{MFE},\small{HMFE}}
    \end{loglogaxis}
  \end{tikzpicture}
  \caption{Test 1, Barry-Mercer's problem: Convergence of the relative $L_2$-error in pressure at $\hat{t}=\pi/2$.}\label{fig:CPMan20}
\end{figure}

Fig.~\ref{fig:CPMan20} shows convergence behavior of the relative $L_2$-error of the pressure solution for both the mixed and mixed hybrid stabilized formulation.
The outcome of the two formulations is the same, with a linear convergence rate, as expected.
Figs.~\ref{fig:BM13} and~\ref{fig:BM1632}  provide a comparison between the stabilized and non stabilized formulations. 
In particular, Fig.~\ref{fig:BM13} shows the contour of the pressure field while Fig.~\ref{fig:BM1632} provides the vertical profiles for two different grid refinement levels.
For each level a zoom is shown to highlight the spurious oscillatory behavior of the unstabilized formulation.
We consider here a very small time step to simulate the process in undrained conditions ($\Delta t=10^{-6}
\pi/(2\hat{\beta})$).
The results show the effectiveness of the stabilization for eliminating the spurious oscillations.

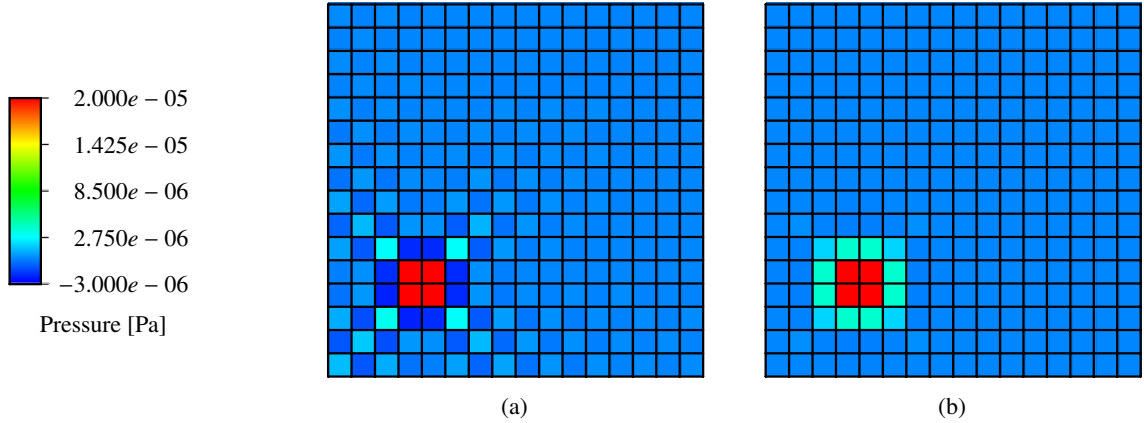
\begin{figure}
  \small
  \centering
  \hfill
  \begin{subfigure}[c]{.25\linewidth}
    \begin{tikzpicture}[scale=1.0]
      \node[anchor=south west,inner sep=0] (image) at (.4\linewidth,0) {\includegraphics[width=.125\linewidth,height=.60\linewidth]{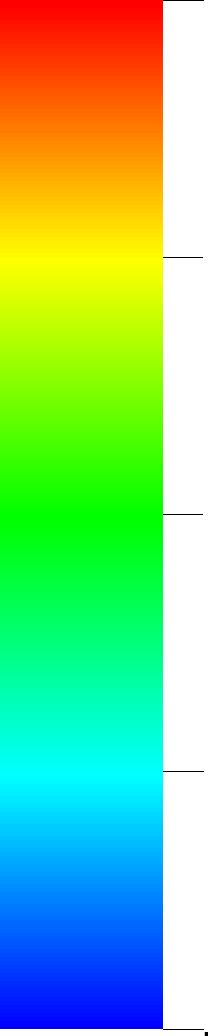}};
      \draw [line width=0.3mm,color=black,line cap=round]  (0.400\linewidth,0.000\linewidth) --
                                                           (0.500\linewidth,0.000\linewidth) --
                                                           (0.500\linewidth,0.600\linewidth) --
                                                           (0.400\linewidth,0.600\linewidth) --
                                                           (0.400\linewidth,0.000\linewidth);
      \foreach \x in {0.0, 0.15, ..., 0.6}
        \draw (0.500*\linewidth, \x*\linewidth) -- (0.525*\linewidth, \x*\linewidth);
      \node [left] at (\linewidth,0.60*\linewidth) {$2.000e-05$};
      \node [left] at (\linewidth,0.45*\linewidth) {$1.425e-05$};
      \node [left] at (\linewidth,0.30*\linewidth) {$8.500e-06$};
      \node [left] at (\linewidth,0.15*\linewidth) {$2.750e-06$};
      \node [left] at (\linewidth,0.00*\linewidth) {$-3.000e-06$};
      \node [below] at (.7\linewidth,-0.075*\linewidth) {Pressure [Pa]};
    \end{tikzpicture}    
  \end{subfigure}
  \begin{subfigure}[c]{.3\linewidth}
    \begin{tikzpicture}[scale=1.0]
      \node[anchor=south west,inner sep=0] (image) at (0,0) {\includegraphics[width=\linewidth]{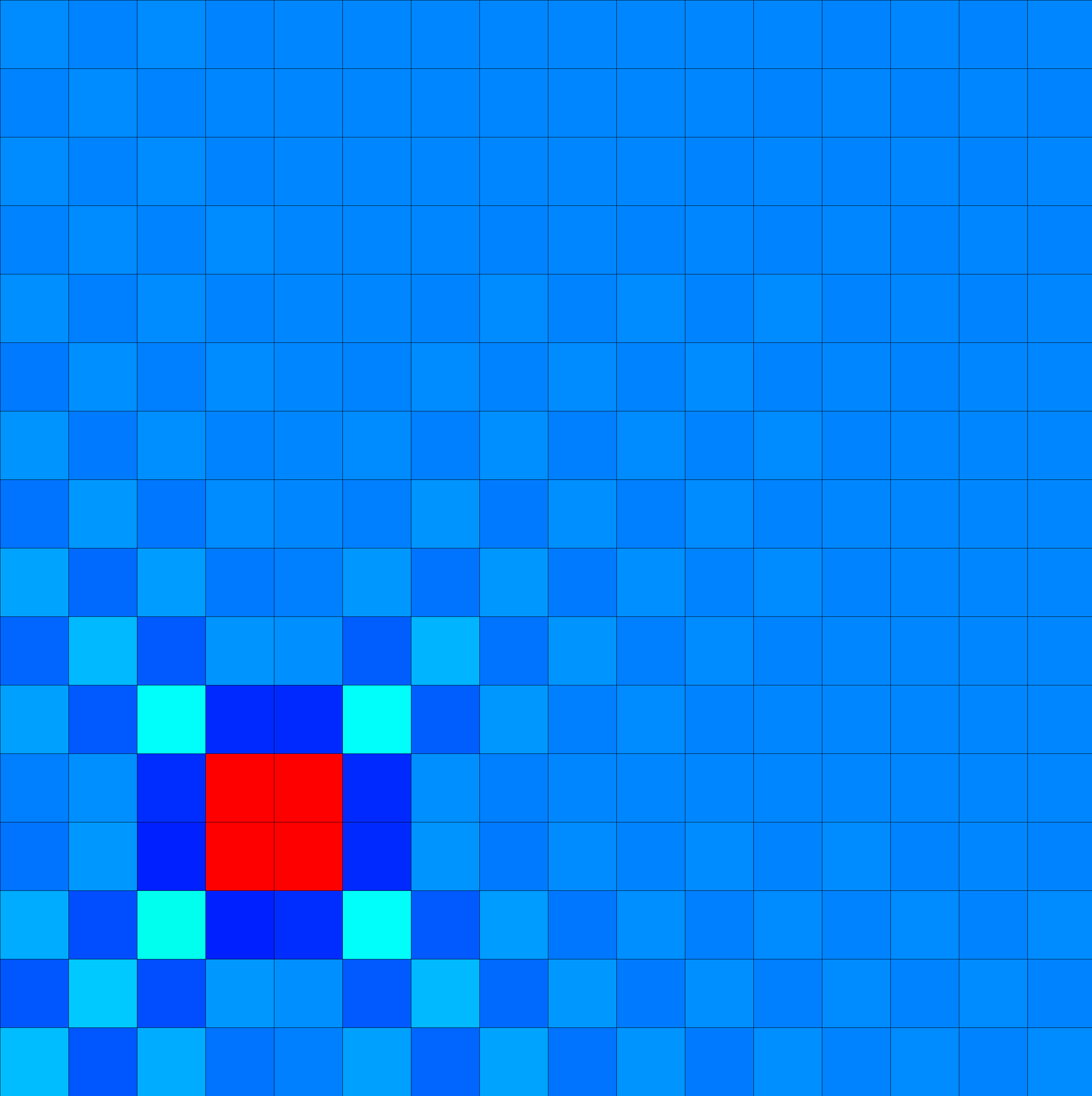}};
      \draw[ step=1*\linewidth/16,
             black,
             line width=0.3mm,
              line cap=round]
        (0.0,0.0) grid (\linewidth, \linewidth);
    \end{tikzpicture}
    \caption{}
  \end{subfigure}
  \hfill
  \begin{subfigure}[c]{.3\linewidth}
    \begin{tikzpicture}[scale=1.0]
      \node[anchor=south west,inner sep=0] (image) at (0,0) {\includegraphics[width=\linewidth]{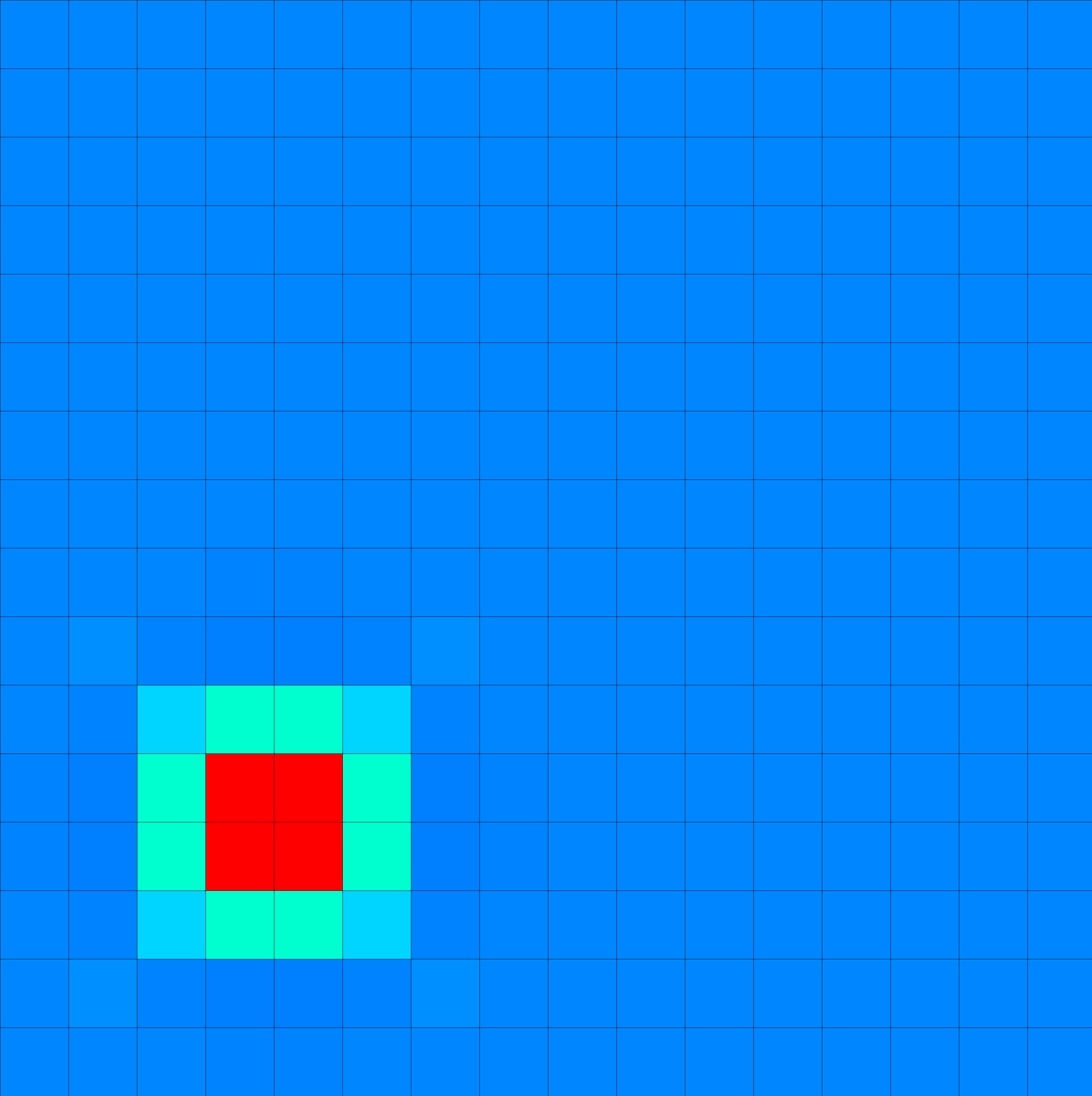}};
      \draw[ step=1*\linewidth/16,
             black,
             line width=0.3mm,
             line cap=round]
        (0.0,0.0) grid (\linewidth, \linewidth);
    \end{tikzpicture}
    \caption{}
  \end{subfigure}
  \hfill\null
\caption{{Test 1, Barry-Mercer's problem: Pressure contour for the unstabilized (a) and stabilized (b) formulations}}\label{fig:BM13}
\end{figure}

\begin{figure}
\centering
\hfill
\begin{subfigure}[c]{.45\linewidth}
\begin{tikzpicture}
\pgfplotsset{scaled y ticks=false}
\begin{axis}[ scale=1.0,width=\linewidth,height=.75\linewidth, grid=major,xmin=0,xmax=1.,ymin=-1.e-5,ymax=4.e-5,
	xlabel={\small{ $y$ [m]} },ylabel={ \small{Pressure [Pa] }},
  xtick={0,0.2,0.4,0.6,0.8,1.0},
  xticklabels={{\small 0},{ \small 0.2},{\small 0.4},{\small 0.6},{\small 0.8},{\small 1.0}},
  ytick={-1e-5,0,1.e-5,2.e-5,3.e-5,4.e-5,5.e-5,6.e-5},
  yticklabels={{\small -1.e-5},{\small 0.0},{\small 1.e-5},{\small 2.e-5 },{\small 3.e-5},
               {\small 4.e-5},{\small 5.e-5}},
  ylabel near ticks,xlabel near ticks]
\addplot [dashed,black,mark=*] table [x=x,y=y1] {./press_16.txt};
\addplot [dashdotted,black,mark=star] table [x=x,y=y2] {./press_16.txt};
\addplot [black,mark=square] table [x=x,y=y3] {./press_16.txt};
\legend{\small{$x=0.094$},\small{$x=0.156$},\small{$x=0.218$}}
\end{axis}
\end{tikzpicture} 
\caption{Unstabilized, $h$ = 1/16}
\end{subfigure}
\hfill
\begin{subfigure}[c]{.45\linewidth}
\begin{tikzpicture}
\pgfplotsset{scaled y ticks=false}
\begin{axis}[ scale=1.0,width=\linewidth,height=.75\linewidth, grid=major,xmin=0,xmax=1.,ymin=-1.e-5,ymax=4.e-5,
	xlabel={\small{ $y$ [m]} },ylabel={ \small{Pressure [Pa] }},
  xtick={0,0.2,0.4,0.6,0.8,1.0},
  xticklabels={{\small 0},{ \small 0.2},{\small 0.4},{\small 0.6},{\small 0.8},{\small 1.0}},
  ytick={-1e-5,0,1.e-5,2.e-5,3.e-5,4.e-5,5.e-5,6.e-5},
  yticklabels={{\small -1.e-5},{\small 0.0},{\small 1.e-5},{\small 2.e-5 },{\small 3.e-5},
               {\small 4.e-5},{\small 5.e-5}},
  ylabel near ticks,xlabel near ticks]
\addplot [dashed,black,mark=*] table [x=x,y=y1] {./press_32.txt};
\addplot [dashdotted,black,mark=star] table [x=x,y=y2] {./press_32.txt};
\addplot [black,mark=square] table [x=x,y=y3] {./press_32.txt};
\legend{\small{$x=0.109$},\small{$x=0.203$},\small{$x=0.234$}}
\end{axis}
\end{tikzpicture} 
\caption{Unstabilized, $h$ = 1/32}
\end{subfigure}
\hfill\null

\hfill
\begin{subfigure}[c]{.45\linewidth}
\begin{tikzpicture}
\pgfplotsset{scaled y ticks=false}
\begin{axis}[ scale=1.0,
              width=\linewidth, height=.75\linewidth,
              grid=major,
              xmin=0, xmax=1., ymin=-1.e-5, ymax=4.e-5,
              xlabel={\small{ $y$ [m]} },ylabel={ \small{Pressure [Pa] }},
              xtick={0,0.2,0.4,0.6,0.8,1.0},
              xticklabels={{\small 0},{ \small 0.2},{\small 0.4},{\small 0.6},{\small 0.8},{\small 1.0}},
              ytick={-1e-5,0,1.e-5,2.e-5,3.e-5,4.e-5,5.e-5,6.e-5},
              yticklabels={{\small -1.e-5},{\small 0.0},{\small 1.e-5},{\small 2.e-5 },{\small 3.e-5},
               {\small 4.e-5},{\small 5.e-5}},
              ylabel near ticks,xlabel near ticks]
\addplot [dashed,black,mark=*] table [x=x,y=y1] {./press_16_stab.txt};
\addplot [dashdotted,black,mark=star] table [x=x,y=y2] {./press_16_stab.txt};
\addplot [black,mark=square] table [x=x,y=y3] {./press_16_stab.txt};
\legend{\small{$x=0.094$},\small{$x=0.156$},\small{$x=0.218$}}
\end{axis}
\end{tikzpicture} 
\caption{Stabilized, $h$ = 1/16}
\end{subfigure}
\hfill
\begin{subfigure}[c]{.45\linewidth}
\begin{tikzpicture}
\pgfplotsset{scaled y ticks=false}
\begin{axis}[ scale=1.0,width=\linewidth,height=.75\linewidth, grid=major,xmin=0,xmax=1.,ymin=-1.e-5,ymax=4.e-5,
	xlabel={\small{ $y$ [m]} },ylabel={ \small{Pressure [Pa] }},
  xtick={0,0.2,0.4,0.6,0.8,1.0},
  xticklabels={{\small 0},{ \small 0.2},{\small 0.4},{\small 0.6},{\small 0.8},{\small 1.0}},
  ytick={-1e-5,0,1.e-5,2.e-5,3.e-5,4.e-5,5.e-5,6.e-5},
  yticklabels={{\small -1.e-5},{\small 0.0},{\small 1.e-5},{\small 2.e-5 },{\small 3.e-5},
               {\small 4.e-5},{\small 5.e-5}},
  ylabel near ticks,xlabel near ticks]
\addplot [dashed,black,mark=*] table [x=x,y=y1] {./press_32_stab.txt};
\addplot [dashdotted,black,mark=star] table [x=x,y=y2] {./press_32_stab.txt};
\addplot [black,mark=square] table [x=x,y=y3] {./press_32_stab.txt};
\legend{\small{$x=0.109$},\small{$x=0.203$},\small{$x=0.234$}}
\end{axis}
\end{tikzpicture} 
\caption{Stabilized, $h$ = 1/32}
\end{subfigure}
\hfill\null
\caption{{Test 1, Barry-Mercer's problem: Pressure along the $y$-axis for the unstabilized (top panels) and stabilized (bottom panels) formulations.}}\label{fig:BM1632}
\end{figure}
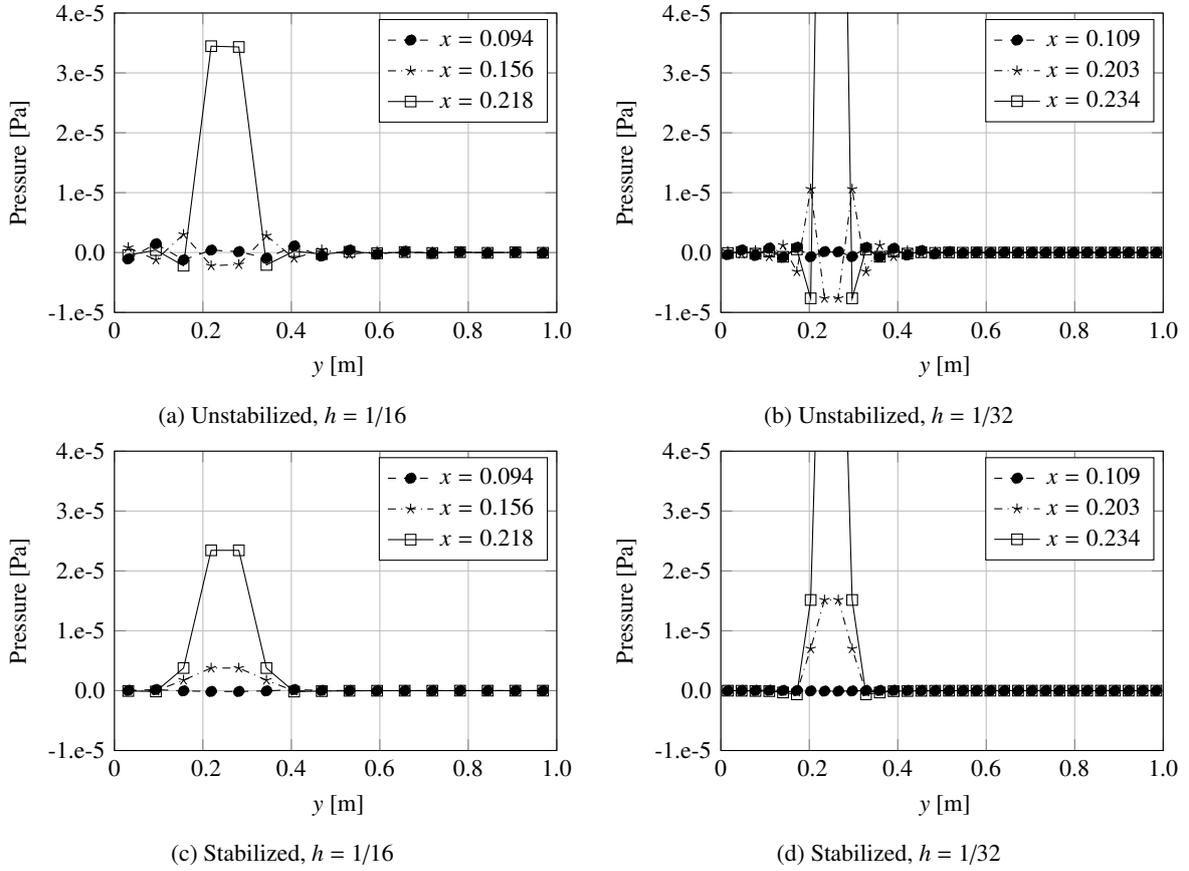

\subsection{Cantilever beam}

\begin{figure}
\centering
  \hfill
  \begin{subfigure}[b]{.45\linewidth}
  \centering
\begin{tikzpicture}[>=latex,font=\small]
    \draw [thick] plot coordinates {(0,3.0) (3.0,3.0)};
    \draw [thick] plot coordinates {(3.0,3.0) (3.0,0.0)};
    \draw [thick] plot coordinates {(3.0,0.0) (0.0,0.0)};
    \draw [thick] plot coordinates {(0.0,0.0) (0.0,3.0)};
    \draw [->] (-1.7,0) -- (-1.7,1.);
    \draw [->] (-1.7,0) -- (-0.7,0);
    \draw [thin] plot coordinates {(-0.03,3.0) (0.03,3.0)};
    \draw [thin] plot coordinates {(3.0,-0.03) (3.0,0.03)};
    \node [left] at (-1.7,1.) {$y$};
    \node [below] at (-0.7,0) {$x$};
    \node [left] at (0,3) { 
    \small {$(0,l)$}};
    \node [below] at (3.,0) { 
    \small{$(l,0)$}};
    \node [below] at (0,0) { 
    \small{$(0,0)$}};
    \fill (0,0) circle (0.2em);
    \fill (3,0) circle (0.2em);
    \fill (0,3) circle (0.2em);
    \node [above] at (0,1.5)  [rotate=90] { 
      \begin{tabular}{@{}c@{}}
        \small{$\tensorOne{u} = \tensorOne{0}$} \\
        \small{$\tensorOne{q} \cdot \tensorOne{n}=\tensorOne{0}$}
      \end{tabular} };    
    \node [below] at (3,1.5)  [rotate=90] { 
      \begin{tabular}{@{}c@{}}
        \small{$\tensorOne{q} \cdot \tensorOne{n}=\tensorOne{0}$} \\
        \small{$\left( \tensorFour{C}_{dr} : \nabla^{s} \tensorOne{u} - b p \tensorTwo{1} \right) \cdot \tensorOne{n}=\tensorOne{0}$}
      \end{tabular} };
    \node [below] at (1.5,0) {
      \begin{tabular}{@{}c@{}}
        \small{$\tensorOne{q} \cdot \tensorOne{n}=\tensorOne{0}$} \\
        \small{$\left( \tensorFour{C}_{dr} : \nabla^{s} \tensorOne{u} - b p \tensorTwo{1} \right) \cdot \tensorOne{n}=\tensorOne{0}$}
      \end{tabular} };
    \node [above] at (1.5,3) {
      \begin{tabular}{@{}c@{}}
        \small{$\tensorOne{q} \cdot \tensorOne{n}=\tensorOne{0}$} \\
        \small{$\left( \tensorFour{C}_{dr} : \nabla^{s} \tensorOne{u} - b p \tensorTwo{1} \right) \cdot \tensorOne{n}=\bar{t} \tensorOne{n}$}
      \end{tabular} };
    \end{tikzpicture}
  \caption{}
  \end{subfigure}
  \hfill
  \begin{subfigure}[b]{.45\linewidth}
    \small
    \centering
    \begingroup
    \renewcommand{\arraystretch}{1.4} 
    \begin{tabular}{lll}
      \toprule
      Quantity                          & Value & Unit \\
      \midrule   
      Young's modulus ($E$)             & $1 \times 10^5$    & [Pa] \\  
      Poisson's ratio ($\nu$)           & $0.4$              & [-] \\ 
      Biot's coefficient ($b$)          & $1.0$              & [-] \\ 
      \begin{tabular}[c]{@{}Sl@{}}
        Constrained specific \\[-7pt]
        storage ($S_{\epsilon}$)
      \end{tabular}                     & $0$                & [Pa] \\
      Isotropic permeability ($\kappa$) & $1 \times 10^{-7}$ & [m$^2$] \\
      Fluid viscosity ($\mu$)           & $1 \times 10^{-3}$ & [Pa $\cdot$ s] \\
      Domain size $x$-$y$ ($l$)         & $1.0$              & [m] \\
      \bottomrule
    \end{tabular}
    \endgroup
    \caption{}
  \end{subfigure}
  \hfill\null
\caption{Test 2, Cantilever beam: (a) domain sketch and (b) physical parameters.}\label{fig:Can_dom}
\end{figure}

A porous cantilever beam problem is now considered \cite{PhiWhe09}.
Domain size and properties are summarized in Fig.~\ref{fig:Can_dom}.
The domain is the unit square or cube for the 2-D and 3-D case, respectively. 
No-flow boundary conditions along all sides are imposed, with the displacements
fixed along the left edge and a uniform load applied at the top. 
Fig.~\ref{fig:can_1} shows the pressure solution obtained with a grid spacing $h=1/10$.
In the unstabilized formulation, checkerboard oscillations arise close to the left constrained edge. 
As in the previous test case, the proposed stabilization eliminates the spurious pressure modes.
This behavior can be better observed along the three vertical profiles provided in Fig.~\ref{fig:can_2}.

\begin{figure}
  \small
  \centering
  \hfill
  \begin{subfigure}[c]{.25\linewidth}
    \begin{tikzpicture}[scale=1.0]
      \node[anchor=south west,inner sep=0] (image) at (.525\linewidth,0) {\includegraphics[width=.125\linewidth,height=.60\linewidth]{./bm_16_legend}};
      \draw [line width=0.3mm,color=black,line cap=round]  (0.525\linewidth,0.000\linewidth) --
                                                           (0.625\linewidth,0.000\linewidth) --
                                                           (0.625\linewidth,0.600\linewidth) --
                                                           (0.525\linewidth,0.600\linewidth) --
                                                           (0.525\linewidth,0.000\linewidth);
      \foreach \x in {0.0, 0.15, ..., 0.6}
        \draw (0.625*\linewidth, \x*\linewidth) -- (0.650*\linewidth, \x*\linewidth);
      \node [left] at (.875*\linewidth,0.60*\linewidth) {$ 3.0$};
      \node [left] at (.875*\linewidth,0.45*\linewidth) {$ 1.5$};
      \node [left] at (.875*\linewidth,0.30*\linewidth) {$ 0.0$};
      \node [left] at (.875*\linewidth,0.15*\linewidth) {$-1.5$};
      \node [left] at (.875*\linewidth,0.00*\linewidth) {$-3.0$};
      \node [below] at (.7\linewidth,-0.075*\linewidth) {Pressure [Pa]};
    \end{tikzpicture}    
  \end{subfigure}
  \begin{subfigure}[c]{.3\linewidth}
    \begin{tikzpicture}[scale=1.0]
      \node[anchor=south west,inner sep=0] (image) at (0,0) {\includegraphics[width=\linewidth]{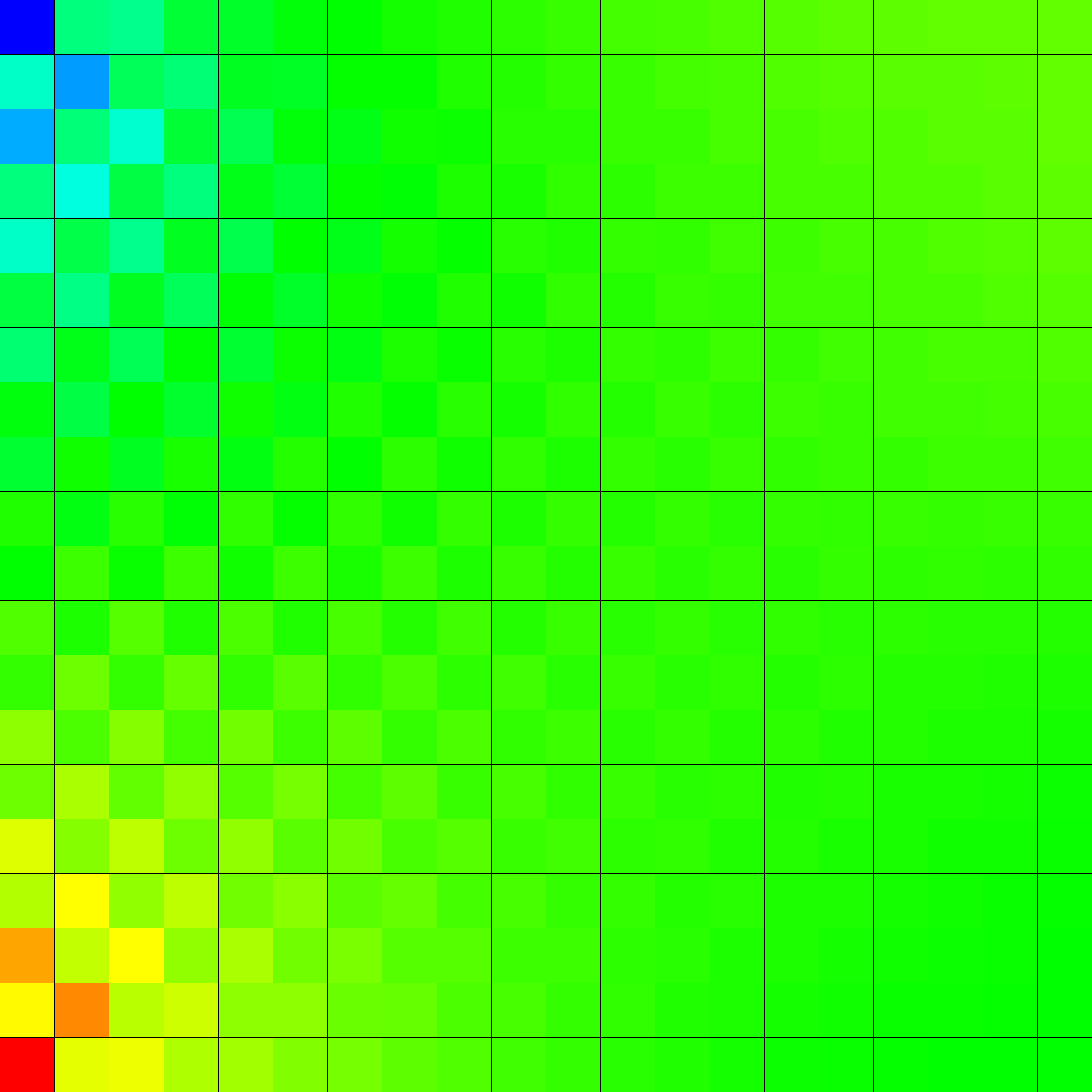}};
      \draw[ step=1*\linewidth/20,
             black,
             line width=0.3mm,
              line cap=round]
        (0.0,0.0) grid (\linewidth, \linewidth);
    \end{tikzpicture}
    \caption{}
  \end{subfigure}
  \hfill
  \begin{subfigure}[c]{.3\linewidth}
    \begin{tikzpicture}[scale=1.0]
      \node[anchor=south west,inner sep=0] (image) at (0,0) {\includegraphics[width=\linewidth]{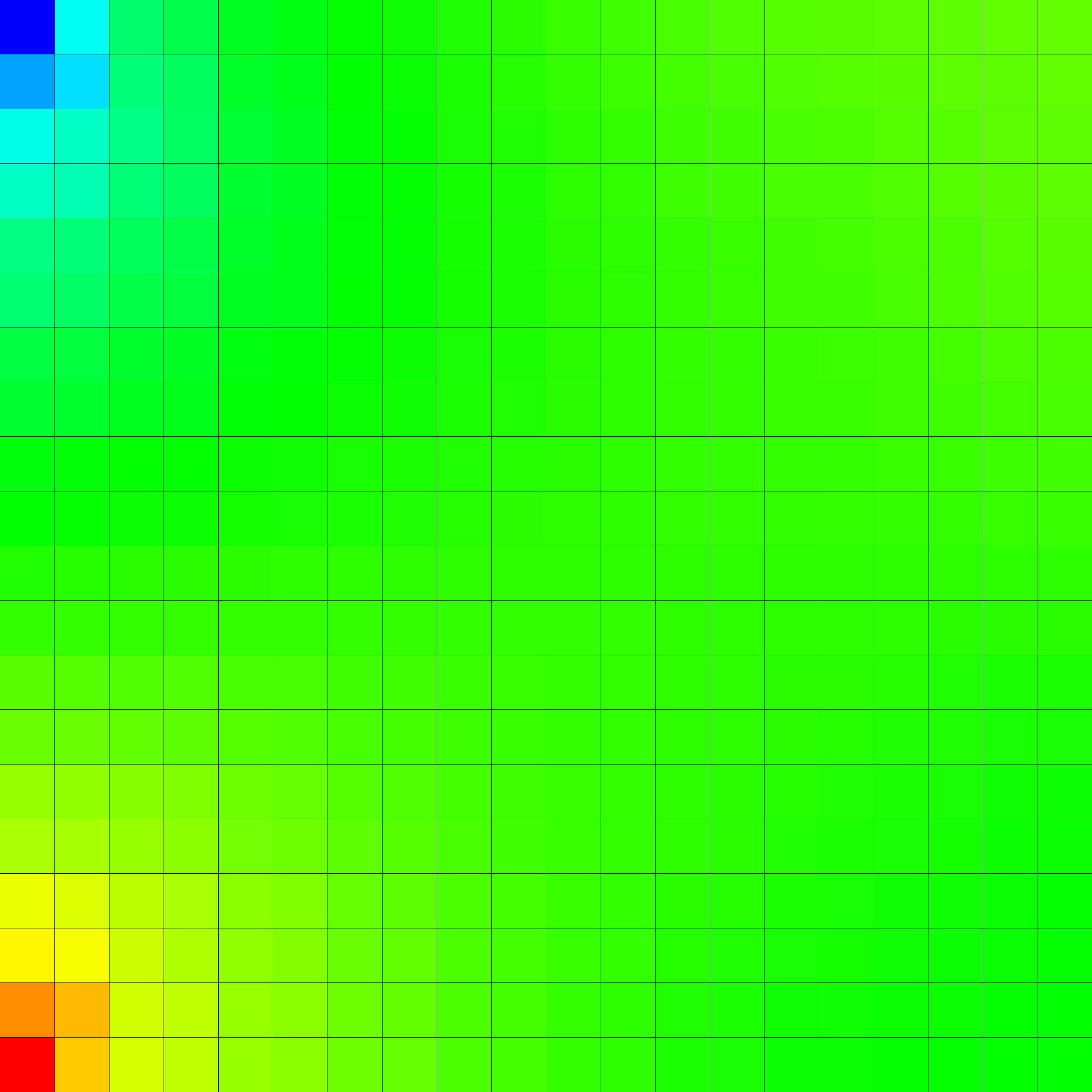}};
      \draw[ step=1*\linewidth/20,
             black,
             line width=0.3mm,
             line cap=round]
        (0.0,0.0) grid (\linewidth, \linewidth);
    \end{tikzpicture}
    \caption{}
  \end{subfigure}
  \hfill\null
\caption{{2D Cantilever beam: pressure solution for the unstabilized (a) and stabilized (b) formulations.}}\label{fig:can_1}
\end{figure}

\begin{figure}
\centering
\hfill
\begin{subfigure}[c]{.45\linewidth}
\begin{tikzpicture}
\pgfplotsset{scaled y ticks=false}
\begin{axis}[ scale=1.0,width=\linewidth,height=.75\linewidth, grid=major,xmin=0,xmax=1.,ymin=-3.,ymax=3.,
	xlabel={{\small $y$ [m]} }, ylabel={ {\small Pressure [Pa] }},
  xtick={0,0.2,0.4,0.6,0.8,1.0},
  xticklabels={{\small 0},{\small 0.2},{\small 0.4},{\small 0.6},
               {\small 0.8},{\small 1.0}},
  ytick={-3,-1.5,0,1.5,3.0},
  yticklabels={{\small -3.0},{\small -1.5},{\small 0.0},
               {\small 1.5},{\small 3.0}},
  ylabel near ticks,xlabel near ticks]
\addplot [dashed,black,mark=*]        table [x=x,y=y1] {./press_10.txt};
\addplot [dashdotted,black,mark=star] table [x=x,y=y2] {./press_10.txt};
\addplot [black,mark=square]          table [x=x,y=y3] {./press_10.txt};
\legend{\small{$x=0.05$},\small{$x=0.15$},\small{$x=0.25$}}
\end{axis}
\end{tikzpicture}	
\caption{Unstabilized}
\end{subfigure}
\hfill
\begin{subfigure}[c]{.45\linewidth}
\begin{tikzpicture}
\pgfplotsset{scaled y ticks=false}
\begin{axis}[ scale=1.0,width=\linewidth,height=.75\linewidth, grid=major,xmin=0,xmax=1.,ymin=-3.,ymax=3.,
	xlabel={{\small $y$ [m]} }, ylabel={ {\small Pressure [Pa] }},
  xtick={0,0.2,0.4,0.6,0.8,1.0},
  xticklabels={{\small 0},{\small 0.2},{\small 0.4},{\small 0.6},
               {\small 0.8},{\small 1.0}},
  ytick={-3,-1.5,0,1.5,3.0},
  yticklabels={{\small -3.0},{\small -1.5},{\small 0.0},
               {\small 1.5},{\small 3.0}},
  ylabel near ticks,xlabel near ticks]
\addplot [dashed,black,mark=*]        table [x=x,y=y1] {./press_10_stab.txt};
\addplot [dashdotted,black,mark=star] table [x=x,y=y2] {./press_10_stab.txt};
\addplot [black,mark=square]          table [x=x,y=y3] {./press_10_stab.txt};
\legend{\small{$x=0.05$},\small{$x=0.15$},\small{$x=0.25$}}
\end{axis}
\end{tikzpicture} 
\caption{Stabilized}
\end{subfigure}
\hfill\null
\caption{{2D Cantilever beam: pressure solution along vertical sections for the unstabilized and stabilized formulations.}}\label{fig:can_2}
\end{figure}
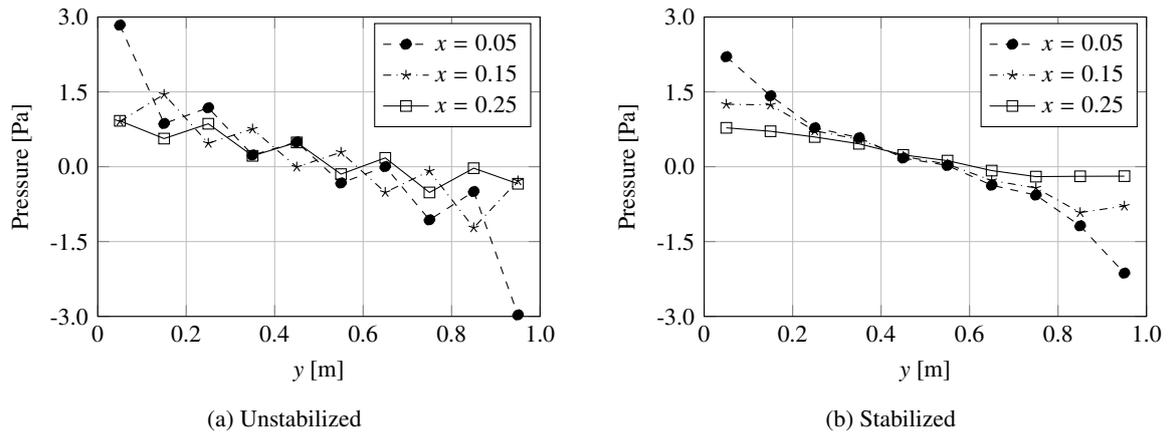

\begin{table}
\centering
\small
\begin{tabular}{ c c r r r r}
\toprule
\multicolumn{1}{c}{$1/ h$}& \begin{tabular}{@{}c@{} } number of  \\ elements \end{tabular} &
\multicolumn{1}{r}{$n_u$}&\multicolumn{1}{r}{$n_q$}&
\multicolumn{1}{r}{$n_p$}&\multicolumn{1}{r}{\begin{tabular}{@{}c@{} } number of  \\ unknowns \end{tabular}} \\ 
\midrule
 10 & $ 10 \times  10 \times  10$ &      3,993 &      3,300 &      1,000 &       8,293 \\ 
 20 & $ 20 \times  20 \times  20$ &     27,783 &     25,200 &      8,000 &      60,983 \\ 
 40 & $ 40 \times  40 \times  40$ &    206,763 &    196,800 &     64,000 &     467,563 \\ 
 64 & $ 64 \times  64 \times  64$ &    823,875 &    798,720 &    262,144 &   1,884,739 \\ 
128 & $128 \times 128 \times 128$ &  6,440,067 &  6,340,608 &  2,097,152 &  14,877,827 \\ 
256 & $256 \times 256 \times 256$ & 50,923,779 & 50,528,256 & 16,777,216 & 118,229,251 \\ 
\bottomrule
\end{tabular} 
\caption{Test 2, 3D Cantilever beam: grid refinement and problem size.
}
\label{tab:grids}
\end{table}

\begin{table}
\centering
\small
\begin{tabular}{c c c c c c c c c c c c c}
\toprule
&&\multicolumn{5}{c}{$\Delta t=0.1$ s}&&\multicolumn{5}{c}{$\Delta t=0.00001$ s}\\
\cline{3-7} \cline{9-13}
&&\multicolumn{2}{c}{ Mixed FE } && \multicolumn{2}{c}{ Hybrid FE } &&\multicolumn{2}{c}{ Mixed FE } && \multicolumn{2}{c}{ Hybrid FE } \\
\cline{3-4} \cline{6-7} \cline{9-10} \cline{12-13}
$1/h$ && No Stab. & Stab. && No Stab. & Stab. && No Stab. & Stab. && No Stab. & Stab. \\
\midrule
10 && 47 & 47 && 42 & 37 && 116 & 49 &&  88 & 39 \\
20 && 52 & 52 && 45 & 41 && 267 & 57 && 161 & 41 \\
40 && 55 & 55 && 48 & 43 && 231 & 63 && 136 & 42 \\
\bottomrule
\end{tabular}
\caption{{Test 2, 3D Cantilever beam: Iteration counts for $\blkMat{M}^{(m)}_I$ and $\blkMat{M}^{(h)}_I$ with the MFE and MHFE formulations.}}\label{tab:can_1}
\end{table}

We analyze the performance of the linear solver and the effects brought by the introduction of the stabilization procedure on 6 successive grid refinements in a 3-D setting, corresponding to the problem size provided in Table~\ref{tab:grids}.
Recall here that $n_q=n_\pi$, which coincides with the number of faces in the grid.
To emphasize the role of the approximations introduced in the Schur complement computations,
we first compare the performance of the block preconditioner variants $\blkMat{M}^{(h)}_{I}$ and $\blkMat{M}^{(m)}_{I}$, i.e., where the inverse of $A_{uu}$, $A_{qq}$, and $\tilde{C}_\pi$ or $\tilde{C}_p$ are applied exactly by nested direct solvers.
Table~\ref{tab:can_1} provides the iteration count for different time-step sizes and the first three grid refinements.
For $\Delta t =0.1$ s, the two formulations give essentially the same outcome.
Indeed, when the conditions are far from the incompressible/undrained limit the effect of the stabilization vanishes as to both the solution accuracy and the solver performance.
On the other hand, with a smaller time-step size, e.g., $\Delta t =0.00001$ s, the preconditioned Krylov convergence behavior can differ significantly between the two formulations. 
An important degradation in the linear solver performance is observed when using the unstable formulation, as a consequence of the presence of near-singular modes.
In the stabilized formulation, such an issue is completely removed and the iteration counts also prove quite stable with the grid size $h$.

For bigger problems, the use of nested direct solvers is no longer viable.
Table~\ref{tab:can_3} shows the performance obtained with $\blkMat{M}^{(m)}_{II}$, $\blkMat{M}^{(h)}_{II}$ and $\blkMat{M}^{(h)}_{III}$, where nested direct solvers are replaced by inner AMG preconditioners, for the same time-step sizes as Table~\ref{tab:can_1} and the finest grids.
In this numerical experiment, the objective is to perform a weak scalability test of the proposed preconditioners, by keeping the same problem size for each processor.
The results show that the mixed hybrid formulation is usually more efficient than the mixed approach, with the $\blkMat{M}^{(h)}_{II}$ preconditioner variant outperforming $\blkMat{M}^{(m)}_{II}$ in all the examined test cases.
The weak scalability of the proposed preconditioners appears to be fairly good, showing only a mild increase of iteration count when refining the grid size up to about 118 million unknowns.
Notice also that all the proposed approaches are optimally scalable with respect to the timestep size, since the number of iterations does not change varying the $\Delta t$ size.

\begin{table}
\centering
\small
\begingroup
\setlength{\tabcolsep}{3pt} 
\begin{tabular}{r r r r r r r r r r r r r r r r r r}
\toprule
&&&&\multicolumn{4}{c}{{Mixed FE ($\blkMat{M}^{(m)}_{II}$)}} && \multicolumn{4}{c}{{Hybrid FE ($\blkMat{M}^{(h)}_{II}$)}} && \multicolumn{4}{c}{{Hybrid FE ($\blkMat{M}^{(h)}_{III}$)}}\\
\cline{5-8} \cline{10-13} \cline{15-18}
{$1/h$}&{\# dofs}&{\# proc.}&&{\# iter.} & {$T_p$ [s]} & {$T_s$ [s]} & {$T_t$ [s]}&&
{\# iter.} & {$T_p$ [s]} & {$T_s$ [s]} & {$T_t$ [s]}&&
{\# iter.} & {$T_p$ [s]} & {$T_s$ [s]} & {$T_t$ [s]}\\
\midrule
{64}&{1,884,739}&{36}&&{55}&{1.0}&{2.9}&{3.9}&&
{53}&{1.0}&{2.5}&{3.5}&&
{36}&{3.4}&{4.8}&{8.2}\\
{128}&{14,877,827}&{288}&&{63}&{1.5}&{4.2}&{5.7}&&
{63}&{1.6}&{3.6}&{5.2}&&
{47}&{4.2}&{6.5}&{10.7}\\
{256}&{118,229,251}&{2268}&&{90}&{2.6}&{8.5}&{11.1}&&
{86}&{2.2}&{7.7}&{9.9}&&
{63}&{6.9}&{9.9}&{16.8}\\
\midrule
{64}&{1,884,739}&{36}&&{56}&{1.0}&{3.2}&{4.2}&&
{52}&{1.0}&{2.7}&{3.7}&&
{31}&{3.4}&{4.3}&{7.7}\\
{128}&{14,877,827}&{288}&&{71}&{1.3}&{5.3}&{6.6}&&
{63}&{1.5}&{4.1}&{5.6}&&
{43}&{3.5}&{6.4}&{9.9}\\
\small{256}&\small{118,229,251}&\small{2268}&&\small{95}&\small{2.8}&\small{12.3}&\small{15.1}&&
\small{83}&\small{2.2}&\small{7.9}&\small{10.1}&&
\small{59}&\small{5.9}&\small{10.1}&\small{16.0}\\
\bottomrule
\end{tabular}
\endgroup
\caption{{Test 2, 3D Cantilever beam: Performance of $\blkMat{M}^{(m)}_{II}$, $\blkMat{M}^{(h)}_{II}$ and $\blkMat{M}^{(h)}_{III}$ for $\Delta t = 0.1$ s (above) and $\Delta t=0.00001$ s (below).}}\label{tab:can_3}
\end{table}

The most expensive effort in all the preconditioner variants relies by far in the inexact solve of $A_{uu}$.
We compare two different V-cycle AMG approaches for this task with the $\blkMat{M}^{(h)}_{II}$  and $\blkMat{M}^{(h)}_{III}$ variants. 
In the former, we introduce the Separate Displacement Component approximation, while in the latter knowledge of the Rigid Body Modes is exploited to build the near-kernel space.
We can observe that in the second case the iteration count to achieve the convergence is reduced by approximately 25\% and 30\%, but such an acceleration does not seem to pay off for the required additional cost.
Therefore, the Separate Displacement Component approach appears to be preferable.

\subsection{SPE10-based benchmark}

\begin{figure}
  \hfill
  \begin{subfigure}[b]{0.45\linewidth}
    \centering
    \includegraphics[width=\linewidth]{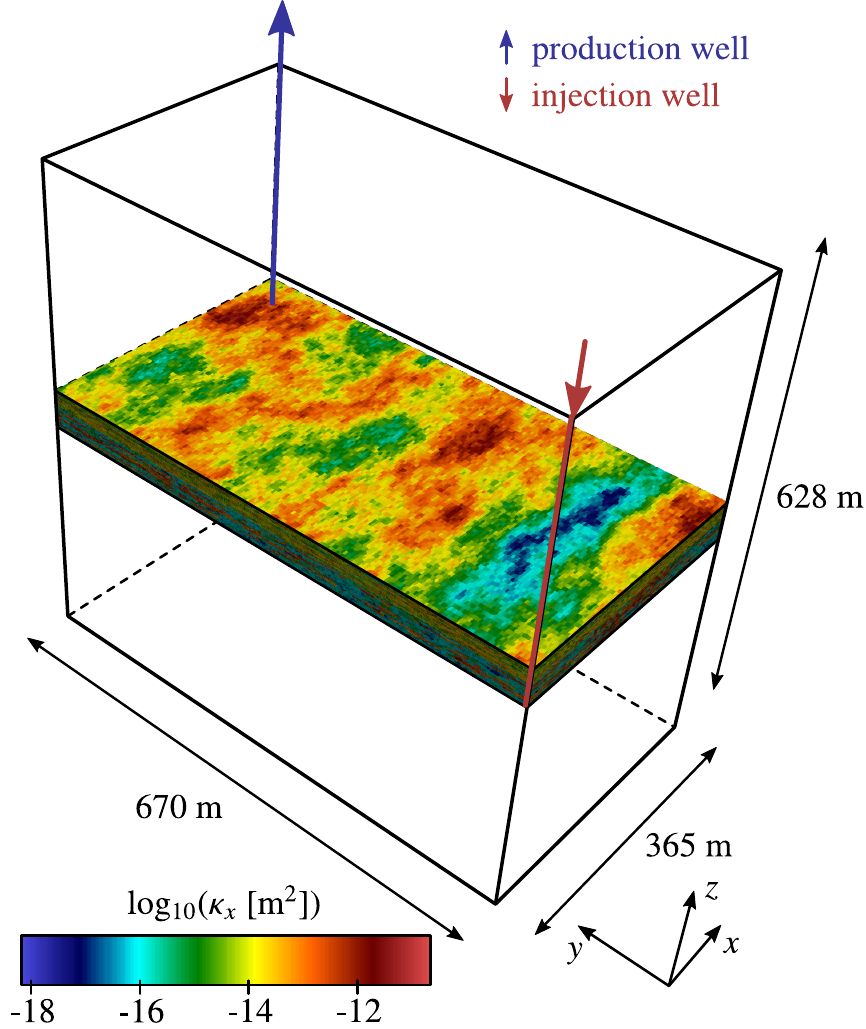}
    \caption{}
  \end{subfigure}
  \hfill
  \begin{subfigure}[b]{.5\linewidth}
    \small
    \centering
    \begingroup
    \renewcommand{\arraystretch}{1.4} 
    \begin{tabular}{lll}
      \toprule
      Quantity                          & Value & Unit \\
      \midrule   
      Young's modulus ($E$)             & $5 \times 10^9$    & [Pa] \\  
      Poisson's ratio ($\nu$)           & $0.25$              & [-] \\ 
      Biot's coefficient ($b$)          & $1.0$              & [-] \\ 
      \begin{tabular}[c]{@{}Sl@{}}
        Constrained specific \\[-7pt]
        storage ($S_{\epsilon}$)
      \end{tabular}                     & $0$                & [Pa] \\
      Reservoir permeability ($\tensorTwo{\kappa}$) & SPE10 data \cite{ChrBlu01} & [m$^2$] \\
      \begin{tabular}[c]{@{}Sl@{}}
        Overburden/underburden \\[-7pt]
        isotropic permeability ($\kappa$)
      \end{tabular}
                                        & $1 \times 10^{-17}$ & [m$^2$] \\
      Fluid viscosity ($\mu$)           & $3 \times 10^{-3}$ & [Pa $\cdot$ s] \\
      Domain size $x$                   & $365$              & [m] \\
      Domain size $y$                   & $670$              & [m] \\
      Domain size $z$                   & $628$              & [m] \\
      \bottomrule
    \end{tabular}
    \endgroup
    \caption{}
  \end{subfigure}
  \hfill\null
\caption{Test 3: SPE10-based benchmark. Sketch of the simulated domain showing the horizontal permeability ($\kappa_x = \kappa_y$) field in the reservoir (a) and hydromechanical parameters (b).}\label{fig:test3}
\end{figure}

For a strong scalability test, we consider a typical petroleum reservoir engineering application reproducing a well-driven flow in a deforming porous medium.
The model setup is based on the 10\textsuperscript{th} SPE Comparative Solution Project~\cite{ChrBlu01}, a well-known, challenging benchmark in reservoir applications.
Here, we add a poroelastic mechanical behavior with incompressible fluid and solid constituents. 
The model adds 288-m  thick  overburden  and  underburden layers  to  the  original  SPE10  reservoir. 
Fig.~\ref{fig:test3} provides a sketch of the physical domain and the relevant hydromechanical properties.
The  computational grid has 3,410,693 nodes, 10,062,960 faces and 3,326,400 cells, for an overall number of degrees of freedom equal to 23,621,439.
We use the  original  SPE10  anisotropic  permeability  distribution~\cite{ChrBlu01}, while an isotropic permeability value equal to 0.01 mD is assigned to the overburden and underburden layers.
Homogeneous Young’s modulus $E=5000$ MPa, Poisson’s ratio $\nu=0.25$, and Biot’s  coefficient $b=1.0$ are  assumed everywhere.
One injector and one production well, located at opposite corners of the domain, penetrate vertically the entire reservoir and drive the porous fluid flow.
The reader can refer to~\cite{Whietal19} for additional details.

The problem is solved using the MHFE formulation and the $\blkMat{M}_{II}^{(h)}$ variant as preconditioning approach.
This SPE10-based benchmark is quite challenging, testing the preconditioner robustness with respect to a strong variability of the permeability and porosity parameters. Mechanical heterogeneity was not introduced because the variation of $E$ and $\nu$ in subsurface applications is typically mild in comparison to other properties. Moreover, 
the mechanical parameters typically influence the overall performance in a marginal way, as shown for instance in~\cite{CasWhiFer16} 
and~\cite{FraCasFer20}.
Regardless, severe jumps in the mechanical parameters can be effectively tackled by improving the quality of the inner
preconditioner approximating the application of $A_{uu}^{-1}$.
We emphasize that restricting the parameters to the limit case of
incompressible solid ($b = 1.0$) and fluid ($S_{\epsilon} = 0.0$) constituents corresponds to the configuration maximizing hydromechanical coupling \cite{KimTchJua11a,CasWhiTch15}, hence the most challenging setup for assessing the overall block preconditioner performance.

Table~\ref{tab:SPE_ss} provides the results of the strong scalability test obtained for solving one linear system with $\Delta t=0.1$ day.
The number of computing processors is progressively doubled while maintaining the same total problem size. 
The iteration count remains nearly constant, with excellent computational efficiency. 

\begin{table}
\centering
\small
\begin{tabular}{ r r r r r r r}
\toprule
{\# proc.}&{\# dofs / proc.}&{\# iter.} 
& {$T_p$ [s]} & {$T_s$ [s]} & {$T_t$ [s]}& Efficiency \\
\midrule
{36}&{656,151}&{85}&{9.8}&{97.3}&{107.1}&{100\%}\\
{72}&{328,075}&{85}&{5.3}&{47.6}&{53.0}&{101\%}\\
{144}&{164,037}&{88}&{3.1}&{22.7}&{25.8}&{104\%}\\
{288}&{82,018}&{89}&{1.8}&{10.4}&{12.2}&{110\%}\\
{576}&{41,009}&{90}&{1.3}&{5.3}&{6.6}&{101\%}\\
\bottomrule
\end{tabular}
\caption{{Test 3, SPE10: Strong scalability test for $\blkMat{M}_{II}^{(h)}$.}}\label{tab:SPE_ss}
\end{table}

\section{Conclusions}

This work presents a three-field (displacement-pressure-Lagrange multiplier) mixed hybrid formulation of coupled poromechanics discretized by low-order elements.
With respect to the mixed approach, the MHFE discretization uses as primary unknown on the element edges or faces a pressure value instead of Darcy's velocity.
This produces a global discrete system with generally better algebraic properties.
Low order spaces, however, such as the $\mathbb{Q}_1-\mathbb{RT}_0-\mathbb{P}_0$ triple, are not inf-sup stable in the limit of undrained/incompressible conditions and might give rise to spurious modes in the pressure solution with a classical checkerboard structure.
A stabilization strategy and an effective solver have been introduced for the mixed hybrid formulation.

The stabilization is based on the macro-element theory and the local pressure jump approach originally introduced for Stokes problems \cite{SilKec90} and more recently for coupled multiphase flow applications \cite{CamWhiBor19}.
It has a number of useful features:
\begin{enumerate}
    \item From the algebraic viewpoint, such a stabilization consists of adding to the $\bar{A}_{pp}$ contribution the matrix $A_{stab}$, whose entries are proportional to an appropriate stabilization parameter. The value of such parameter is automatically selected at the macro-element level such that the limits of the non-zero eigenspectrum of the resulting local Schur complement do not change.
    \item The sparsity pattern of $A_{stab}$ is a subset of that of $A_{p\pi}A_{\pi p}$, therefore the matrix form of the stabilized mixed hybrid formulation is not structurally different from the unstabilized one and does not require any specific modification at the solver level.
    \item The stabilization effectiveness has been validated in two test cases, demonstrating the preservation of the expected convergence rate for the original formulation and an overall improvement of the computational efficiency near undrained/incompressible conditions. 
\end{enumerate}

The convergence of Krylov subspace methods for the solution of the resulting system of linear equations is accelerated by a block triangular preconditioner based on a two-level Schur complement-approximation approach.
Theoretical and computational properties of the proposed algorithm have been investigated, providing the main results that follow:
\begin{enumerate}
    \item A bound for the eigenvalues of the preconditioned matrix has been introduced depending only on the quality of the approximation of the first-level Schur complement $B_p$.
    \item A proof is provided stating that algebraically simple approximations of $B_p$, such as a diagonal matrix based on the fixed-stress splitting approach \cite{CasWhiTch15,CasWhiFer16}, guarantee that the second-level Schur complement $\tilde{C}_\pi$ is SPD with a bounded condition number independently of the time step size $\Delta t$ and the material parameters.
    \item The computational performance of the linear solver, including weak and strong scalability in massively parallel architectures, has been verified in both theoretical benchmarks and field applications totaling up to about 118 millions unknowns. The numerical results show that the proposed solver is: (i) robust with respect to material heterogeneity and anisotropy; (ii) optimally and nearly-optimally scalable vs the timestep and space discretization size, respectively; (iii) strongly scalable in parallel architectures down to about 40,000 unknowns per computing nodes; and (iv) generally more efficient than existing approaches for stabilized mixed three-field formulations. 
\end{enumerate}

\section*{Acknowledgements}
Partial funding was provided by Total S.A. through the FC-MAELSTROM Project.
The authors wish to thank Chak Lee for helpful discussions.
Portions of this work were performed by MF and MF within the 2020 INdAM-GNCS project ``Optimization and advanced linear algebra for PDE-governed problems''. Portions of this work were performed by NC and JAW under the auspices of the U.S. Department of Energy by Lawrence Livermore National Laboratory under Contract DE-AC52-07NA27344.

\appendix
\section{Finite Element matrices and vectors}\label{app:mat}
The matrices and vectors introduced in Section \ref{sec:MFE} are assembled in the standard way from the elemental contributions. In \eqref{eq:mfeSys}, the global matrix expressions read:
\begin{linenomath}
\begin{subequations}
\begin{align}
  {[A_{uu}]}_{ij}&={(\nabla^s \tensorTwo{\eta}_i, \tensorFour{C}_{\text{dr}} : \nabla^s \tensorTwo{\eta}_j)}_\Omega, &
  &&
  {[A_{up}]}_{ij}&=-{(\text{div} \; \tensorTwo{\eta}_i, b \chi_j)}_\Omega, \\
  &&
  {[A_{qq}]}_{ij}&={( \tensorTwo{\phi}_i, \mu \tensorTwo{\kappa}^{-1} \cdot \tensorTwo{\phi}_j)}_\Omega, &
  {[A_{qp}]}_{ij}&=-{(\text{div} \; \tensorTwo{\phi}_i, \chi_j)}_\Omega,\\
  {[A_{pu}]}_{ij}&={(\chi_i, b \text{div} \; \tensorTwo{\eta}_j)}_\Omega, &
  {[A_{pq}]}_{ij}&={(\chi_i, \text{div} \; \tensorTwo{\phi}_j)}_\Omega, &
  {[A_{pp}]}_{ij}&={(\chi_i, S_{\epsilon} \chi_j)}_\Omega,
\end{align}
\label{eq:MFE_matrix}\null
\end{subequations}
\end{linenomath}
while the global right-hand side vectors are:
\begin{linenomath}
\begin{subequations}
\begin{align}
    {[\vec{f}_u]}_i&= {(\tensorTwo{\eta}_i \cdot \bar{\tensorOne{t}}_n)}_{\Gamma_\sigma} - {(\nabla^s \tensorTwo{\eta}_i, \tensorFour{C}_{\text{dr}} : \nabla^s \bar{\tensorOne{u}}^h_n)}_\Omega, \\
  {[\vec{f}_{q}]}_i&=- {(\tensorTwo{\phi}_i \cdot \tensorOne{n}, \bar{p}_n )}_{\Gamma_p}  - {(\tensorTwo{\phi}_i, \mu \tensorTwo{\kappa}^{-1} \cdot \bar{\tensorOne{q}}^h_n)}_\Omega, \\
  {[\vec{f}_p]}_i&= {(\chi_i, \tilde{f}_n)}_\Omega - {(b \; \text{div} \; \bar{\tensorOne{u}}^h_n, \chi_i )}_\Omega - \Delta t \; {(\; \text{div} \; \bar{\tensorOne{q}}^h_n, \chi_i )}_\Omega.
\end{align}
\label{eq:MFE_vector}\null
\end{subequations}
\end{linenomath}
The additional matrices and vectors introduced in Eqs. \eqref{eq:hmfeSys}-\eqref{eq:A_33H} read:
\begin{linenomath}
\begin{subequations}
\begin{align}
  {[A_{ww}]}_{ij}&={(\tensorTwo{\varphi}_i, \mu \tensorTwo{\kappa}^{-1} \cdot \tensorTwo{\varphi}_j)}_\Omega, &
  {[A_{wp}]}_{ij}&=-\sum_{T \in \mathcal{T}_h} {(\text{div} \; \tensorTwo{\varphi}_i, \chi_j)}_{T}, &
  {[A_{w\pi}]}_{ij}&=\sum_{T \in \mathcal{T}_h} {(\tensorTwo{\varphi}_i \cdot \tensorOne{n}_e, \zeta_j)}_{\partial T}, \\
  {[A_{pw}]}_{ij}&=\sum_{T \in \mathcal{T}_h} {(\chi_i, \text{div} \; \tensorTwo{\varphi}_j)}_{T}, \\
  {[A_{\pi w}]}_{ij}&=-\sum_{T \in \mathcal{T}_h} {(\zeta_i, \tensorTwo{\varphi}_j \cdot \tensorOne{n}_e)}_{\partial T},
  \end{align}
\label{eq:MHFE_matrix_OLD}\null
\end{subequations}
\end{linenomath}
and
\begin{linenomath}
\begin{subequations}
\begin{align}
  {[\vec{f}_w]}_i&= - {(\tensorTwo{\varphi}_i \cdot \tensorOne{n}, \bar{\pi}^h_n)}_{\Gamma_p}, \\
  {[\vec{f}_{p,H}]}_i&= {(\chi_i, \tilde{f}_n)}_\Omega - {(b \; \text{div} \; \bar{\tensorOne{u}}^h_n, \chi_i )}_\Omega, \\
  {[\vec{f}_{{\pi}}]}_i&=-{(\zeta_i,\bar{q}_n)}_{\Gamma_q}. 
\end{align}
\label{eq:MHFE_vector}\null
\end{subequations}
\end{linenomath}
Finally, the stabilization matrix introduced in \eqref{eq:stab_A} for both the MFE and MHFE discrete formulations is:
\begin{linenomath}
\begin{align}
{[A_{stab}]}_{ij}&=\sum_{M\in\mathcal{M}_h} \beta_M |M| \sum_{e \in \Gamma_M} \llbracket \chi_i \rrbracket_e \llbracket \chi_j \rrbracket_e.
\end{align}
\end{linenomath}
In the expressions above, $\{\tensorTwo{\eta}_i, \tensorTwo{\eta}_j \}$, $\{ \tensorTwo{\phi}_i, \tensorTwo{\phi}_j \}$, $\{ \tensorTwo{\varphi}_i, \tensorTwo{\varphi}_j \}$, $\{ \chi_i, \chi_j\}$ and $\{ \zeta_i, \zeta_j\}$ range over the bases for $\boldsymbol{\mathcal{U}}^h_0$, $\boldsymbol{\mathcal{Q}}^h_0$, $\boldsymbol{\mathcal{W}}^h$, $\mathcal{P}^h$ and $\mathcal{B}^h_0$ respectively.

\bibliography{authorList,journalAbbreviationExtended,biblio}

\end{document}

%% file: packagesAndMacros.tex
\usepackage[utf8]{inputenc}

\usepackage{amssymb,amsmath,amsfonts,mathtools}
\usepackage{amsthm}
	\theoremstyle{remark}
	\newtheorem{rem}{Remark}[section]

	\theoremstyle{plain}
	\newtheorem{thm}{Theorem}[section]
	
	\newtheorem{lem}[thm]{Lemma}

\usepackage{graphicx}
\usepackage{graphbox}
\usepackage[colorinlistoftodos]{todonotes}
\usepackage[colorlinks=true, allcolors=blue]{hyperref}
\usepackage{algorithm}
\usepackage{algpseudocode}
\usepackage{subcaption}
\usepackage{booktabs,cellspace}
\usepackage{multirow}
\usepackage{makecell}

\usepackage{tikz}
\usepackage{pgfplots}
\usetikzlibrary{patterns}
\pgfplotsset{compat=newest}
\usetikzlibrary{calc}

\usepackage{natbib}
\biboptions{sort&compress}
\usepackage{lineno,hyperref}
\usepackage{hypernat}
\modulolinenumbers[5]
\usepackage[]{setspace}

\usepackage{todonotes}

\usepackage{listings}
\usepackage{color} 
\definecolor{mygreen}{RGB}{28,172,0} 
\definecolor{mylilas}{RGB}{170,55,241}

\renewcommand{\vec}[1]{\mathbf{#1}}

\newcommand{\tensorOne}[1]{\boldsymbol{#1}}
\newcommand{\tensorTwo}[1]{\boldsymbol{#1}}
\newcommand{\tensorFour}[1]{\mathbb{#1}}

\newcommand{\Mat}[1]{#1}
\renewcommand{\Vec}[1]{\boldsymbol{\mathbf{#1}}}
\newcommand{\blkMat}[1]{\boldsymbol{\mathbf{#1}}}
\newcommand{\blkVec}[1]{\boldsymbol{\mathbf{#1}}}

\newcommand{\sub}[1]{_\mathit{#1}}

\newcommand{\logLogSlopeTriangle}[5]
{

    \pgfplotsextra{
        \pgfkeysgetvalue{/pgfplots/xmin}{\xmin}
        \pgfkeysgetvalue{/pgfplots/xmax}{\xmax}
        \pgfkeysgetvalue{/pgfplots/ymin}{\ymin}
        \pgfkeysgetvalue{/pgfplots/ymax}{\ymax}

        \pgfmathsetmacro{\xArel}{#1}
        \pgfmathsetmacro{\yArel}{#3}
        \pgfmathsetmacro{\xBrel}{#1-#2}
        \pgfmathsetmacro{\yBrel}{\yArel}
        \pgfmathsetmacro{\xCrel}{\xArel}

        \pgfmathsetmacro{\lnxB}{\xmin* (1- (#1-#2))+\xmax* (#1-#2)} 
        \pgfmathsetmacro{\lnxA}{\xmin* (1-#1)+\xmax*#1} 
        \pgfmathsetmacro{\lnyA}{\ymin* (1-#3)+\ymax*#3} 
        \pgfmathsetmacro{\lnyC}{\lnyA+#4* (\lnxA-\lnxB)}
        \pgfmathsetmacro{\yCrel}{(\lnyC-\ymin)/ (\ymax-\ymin)} 

        \coordinate(A) at (rel axis cs:\xArel,\yArel);
        \coordinate(B) at (rel axis cs:\xBrel,\yBrel);
        \coordinate(C) at (rel axis cs:\xCrel,\yCrel);

        \draw[#5]   (A)-- node[pos=0.5,anchor=north] {1}
                    (B)-- 
                    (C)-- node[pos=0.5,anchor=west] {#4}
                    cycle;
    }
}